\documentclass[a4paper]{article}

\usepackage[english]{babel}
\usepackage[utf8]{inputenc}
\usepackage{amsmath,amsthm,multirow}
\usepackage{graphicx}
\usepackage{caption}
\usepackage{subcaption}
\usepackage{nameref}
\usepackage[colorinlistoftodos]{todonotes}
\usepackage{amssymb}
\usepackage{xcolor}
\usepackage[sc]{titlesec}
\usepackage{xspace} 

\usepackage{booktabs} 

\usepackage{cite} 

\usepackage{fullpage}

\usepackage{mathtools} 
\usepackage{xargs} 

\definecolor{lgreen}{rgb}{0.0, 0.48, 0.0}
\definecolor{lpurple}{rgb}{0.48, 0.0, 0.48}

\usepackage[breaklinks]{hyperref}
\usepackage[
    capitalize,
    nameinlink,
    noabbrev
]{cleveref}
\usepackage{url}
\hypersetup{linktocpage,
            colorlinks=true,
            linkcolor=lgreen,
            citecolor=lpurple}

\definecolor{bblue}{rgb}{0.2, 0.4, 0.8}
\definecolor{bgreen}{rgb}{0.2, 0.6, 0.4}
\definecolor{bred}{rgb}{0.8, 0.4, 0.2}
\definecolor{bviolet}{rgb}{0.7, 0.2, 0.7}
\definecolor{blackred}{rgb}{0.6, 0.3, 0.3}
\definecolor{blackblue}{rgb}{0.3, 0.3, 0.6}
\definecolor{byellow}{rgb}{0.8, 0.6, 0.0}
\definecolor{borange}{rgb}{0.8, 0.4, 0.2}

\makeatletter
\newcommand{\oset}[3][0ex]{%
  \mathrel{\mathop{#3}\limits^{
    \vbox to#1{\kern-2\ex@
    \hbox{$\scriptstyle#2$}\vss}}}}
\makeatother

\def \n {\overline}

\newcommand{\fn}[1]{%
    \oset[-.3ex]{%
        \rule{5pt}{1pt}%
    }{%
        #1%
    }%
}

\usepackage{tikz}
\usetikzlibrary{fit,arrows,trees,shapes,shapes.geometric,calc,decorations.markings}
\tikzset{
  treenode/.style = {align=center, inner sep=0pt, text centered,
    font=\sffamily},
  arnBleuPetit/.style = {treenode, circle, black, draw=bblue,
    fill=bblue!08,
    minimum width=0.8em, minimum height=0.5em
  },
  arnRougePetit/.style = {treenode, circle, black, draw=bred,
    fill=bred!20,
    minimum width=0.8em, minimum height=0.5em
  },
  arnVertPetit/.style = {treenode, circle, black, draw=bgreen,
    fill=bgreen!20,
    minimum width=0.8em, minimum height=0.5em
  },
  arnVioletPetit/.style = {treenode, circle, black, draw=bviolet,
    fill=bviolet!08,
    minimum width=0.8em, minimum height=0.5em
  },
  arnBleuGrande/.style = {treenode, circle, bblue, draw=bblue,
    text width=1.5em, very thick,
    fill=bblue!10},
  arnVioletGrande/.style = {treenode, circle, bviolet, draw=bviolet,
    text width=1.5em, very thick,
    fill=bblue!10},
  arnOrangeGrande/.style = {treenode, circle, borange, draw=borange,
    text width=1.5em, very thick,
    fill=borange!10},
  arnVertGrande/.style = {treenode, circle, bgreen, draw=bgreen,
    text width=1.5em, very thick,
    fill=bgreen!10},
  arn_nn/.style = {treenode, circle, bblue, draw=bblue,
    fill=bblue!10,
    minimum width=0.9em, minimum height=0.9em
  },
  arn_rr/.style = {treenode, circle, bred, draw=bred,
    fill=bred!10,
    minimum width=0.9em, minimum height=0.9em
  },
  arn_n/.style = {treenode, circle, black, draw=bblue,
    text width=1.25em, very thick,
    fill=bblue!08},
  arn_b/.style = {treenode, circle, black, draw=bblue, 
    text width=1.25em, very thick,
    fill=bblue!08},
  arn_r/.style = {treenode, circle, black, draw=bred, 
    text width=1.15em, very thick,
    fill=red!08},
  arn_v/.style = {treenode, circle, bviolet, draw=bviolet, 
    text width=1.5em, very thick,
    fill=bviolet!10},
  arn_g/.style = {treenode, circle, bgreen, draw=bgreen, 
    text width=1.5em, very thick,
    fill=bblue!10},
  arn_y/.style = {treenode, circle, byellow, draw=byellow, 
    text width=1.5em, very thick,
    fill=byellow!10},
  arn_x/.style = {treenode, triangle, draw=black,
    minimum width=0.5em, minimum height=0.5em},
  triangle/.style = {treenode, bred, draw=bred,
      fill=bred!20, regular polygon, regular polygon
    sides=3, very thick, text width=1.5em },
}

\newcommandx*{\fromto}[7][3=black,4=black,5=0,6=0,7=.995,usedefault]{
    \path (#1) edge [thick,#4,bend right=#5,
    decoration={markings,mark=at position #7 with
    {\arrow[thick,#3, rotate=-#6]{>}}}, postaction={decorate}
    ] node {} (#2);
}

\newcommand{\aprod}[2]{
    \path (#1) edge [blackred,thick,
    decoration={markings,mark=at position 1 with
    {\arrow[ultra thick,blackred, rotate=0]{>}}}, postaction={decorate}
    ] node {} (#2);
}

\theoremstyle{definition}

\newtheorem{theorem}{Theorem}[section]
\newtheorem{remark}[theorem]{Remark}
\newtheorem{proposition}[theorem]{Proposition}
\newtheorem{lemma}[theorem]{Lemma}
\newtheorem{definition}[theorem]{Definition}

\newtheorem{example}[theorem]{Example}

\newcommand*{\eg}{\textit{e.g.}\@\xspace}
\newcommand*{\ie}{\textit{i.e.}\@\xspace}
\newcommand*{\cf}{\textit{c.f.}\@\xspace}

\newcommand*{\resp}{resp.\@\xspace}

\def \n   {\overline}

\def \proba {\mathbb{P}}

\newcommand{\graphic}[1]{\mathbf{#1}}
\newcommand{\gset}{\widehat{\graphic{Set}}}
\newcommand{\idset}{\ddot{\graphic{Set}}}

\newcommand{\SCC}{\mathrm{SCC}}
\newcommand{\CSCC}{\mathrm{CSCC}}
\newcommand{\scc}{\mathrm{Scc}}
\newcommand{\cscc}{\mathrm{Cscc}}
\newcommand{\CNF}{\mathrm{CNF}}
\newcommand{\SAT}{\mathrm{SAT}}
\newcommand{\UNSAT}{\mathrm{UNSAT}}
\newcommand{\True}{\mathrm{True}}
\newcommand{\False}{\mathrm{False}}
\newcommand{\mA}{\mathcal{A}}
\newcommand{\mB}{\mathcal{B}}
\newcommand{\mC}{\mathcal{C}}
\newcommand{\mD}{\mathcal{D}}
\newcommand{\mF}{\mathcal{F}}
\newcommand{\bigO}{O}

\let \leq \leqslant
\let \geq \geqslant

\newcommand{\point}[4]{
  \draw
  node[arn_nn](#1a) at (#2, #3) {  }
  node[arn_rr](#1b) at (#2, #3+0.5) { }
  ;
  \node[rectangle,dashed,draw,fit=(#1a)(#1b),
    rounded corners=3mm,inner sep=3pt, bgreen] {};
  \node at (#2+#4,#3) {$ {#1} $};
  \node at (#2+#4,#3+0.5) {$\n {#1}$};
 }

\newcommand{\hpoint}[4]{
  \draw
  node[arn_nn](#1a) at (#2, #3) {  }
  node[arn_rr](#1b) at (#2+0.5, #3) { }
  ;
  \node[rectangle,dashed,draw,fit=(#1a)(#1b),
    rounded corners=3mm,inner sep=3pt, bgreen] {};
  \node at (#2,#3+#4) {$ {#1} $};
  \node at (#2+0.5,#3+#4) {$\n {#1}$};
 }

 \newcommandx*{\sedge}[4][3=blackblue,4=0,usedefault]{
    \path (#1) edge [thick,#3,bend right=#4] node {} (#2);
}

\title{Exact enumeration of satisfiable 2-SAT formulae}

\author{Sergey Dovgal$^{1}$, \'Elie de Panafieu$^{2}$ and Vlady Ravelomanana$^{3}$}
\date{
  \footnotesize{
    $^{1}${LaBRI, CNRS UMR 5800. Université de Bordeaux.
    \href{mailto:vit.north@gmail.com}{\textsf{dovgal.alea-at-gmail.com}}}\\
    $^{2}${Nokia Bell Labs.
    \href{mailto:elie.de_panafieu@nokia-bell-labs.com}{\textsf{elie.de\_panafieu-at-nokia-bell-labs.com}}}\\
    $^{3}${IRIF, CNRS UMR 8243. Université de Paris.
    \href{mailto:vlad@irif.fr}{\textsf{vlad-at-irif.fr}}}%
    \footnote{Authors presented in alphabetical order.}\\
  }
  \today
}

\begin{document}

\maketitle

\begin{abstract}
We obtain exact expressions counting the satisfiable 2-SAT formulae and describe the structure of associated implication digraphs.
Our approach is based on generating function manipulations.
To reflect the combinatorial specificities of the implication digraphs, we introduce a new kind of generating function, the Implication generating function, inspired by the Graphic generating function used in digraph enumeration.
Using the underlying recurrences, we make accurate numerical predictions of the phase transition curve of the 2-SAT problem inside the critical window.
We expect these exact formulae to be amenable to rigorous asymptotic analysis using complex analytic tools, leading to a more detailed picture of the 2-SAT phase transition
in the future.
\end{abstract}

\tableofcontents

\section{Introduction}

A $k$-CNF (Conjunctive Normal Form), or $k$-SAT formula,
is a collection (conjunction) of clauses,
where each clause is a disjunction of Boolean literals,
and each Boolean literal is either a Boolean variable $x$
or its negation $\n x$
(\eg \( (x_1 \lor x_2 \lor \n x_3) \land (\n x_1 \lor x_3 \lor x_4) \)
for a $3$-CNF).
A $k$-CNF is \emph{satisfiable}
if it returns \( \True \) for at least one instantiation
of its variables.
More formal definitions are postponed until \cref{sec:definitions:notations}.

The problem $k$-SAT of deciding whether a $k$-CNF is satisfiable
is a central Constraint Satisfaction Problem (CSP for short).
It is NP-complete for $k \geq 3$ \cite{Coo71},
but $2$-SAT is solvable in polynomial
\cite{krom1967decision, even1975complexity}.
and even in linear~\cite{APT79} time.
This linear algorithm for 2-SAT solving relies on the so-called
\emph{implication digraphs} which are digraphs built from the 2-CNF
where each of the clauses \( (x \lor y) \) is replaced with two directed
edges \( \n x \to y \) and \( \n y \to x \), which corresponds to a logically
equivalent replacement of disjunction by an implication (see further
discussion in \cref{sec:structural:properties:cnf}).
Although $2$-SAT is simple from a computational perspective,
some of its variants are difficult.
For example, the problem of counting the number of solutions
of a $2$-SAT formula is NP-complete \cite{garey1974some},
and even approximating it within a factor less than $4/3$
is NP-complete~\cite{hastad2001}.


\bigskip

The main contribution of the current paper
is an exact expression for the number of satisfiable $2$-CNF
with a given number of variables and clauses
(other results are summarized in \cref{sec:our:results}).
Our enumeration is expressed using generating functions
and our tools are influenced by directed graph (digraph for short)
enumeration.
Before stating our results on the enumeration of satisfiable $2$-CNF
in~\cref{sec:our:results},
the three next subsections provide some historical background
and discuss various angles of attack for the problem.

    \subsection{Exact enumeration}

\paragraph{Graphs.}
Graphical enumeration (see Harary and Palmer~\cite{HararyPalmer})
is a classical theme
in enumerative combinatorics.
It started in the late 19th century,
with the enumeration of labeled trees
by Borchardt~\cite{borchardt1861uber}
and Cayley~\cite{Cayley}.
Around a hundred years later,
in a series of important papers,
Wright~\cite{Wright77,Wright78,Wright80}
obtained exact expressions for the number of connected graphs
according to their number of vertices and edges.
Those expressions rely on the framework of generating functions,
and excellent introductions to those tools are provided
by Bergeron, Labelle, and Leroux~\cite{bergeron1998combinatorial},
and Flajolet and Sedgewick~\cite{flajolet2009analytic}.

\paragraph{Digraphs.}
In the 1970's, Liskovets, Robinson and Wright~\cite{Wright71,liskovets17contribution,Robinson71}
developed the enumeration of strongly connected digraphs,
directed acyclic graphs, and related digraph families.
Most importantly, Robinson was able to lift the enumeration
from the level of recurrences to the level of generating functions
and obtain the tools to handle very general digraph families.
Later, Robinson and Liskovets extended these methods
to unlabeled enumeration.
A detailed historical account on the development of
the directed enumeration can be found in
the introduction of~\cite{de2019symbolic},
where this method has been again rediscovered.

\paragraph{2-CSP.}

A CSP formula is similar to a CNF formula, but more logical operators
are allowed, such as the XOR operation $\oplus$ and logical implication
$\to$, further expanding to Boolean functions of more than two
variables. 

As various communities, namely computer scientists, probabilists and
physicists
\cite{achlioptas2008algorithmic,bollobas2001scaling,monasson1997statistical},
have been interested in CSP, in what follows we propose to present the
line of research and the links between them that led us to approach the
$2$-SAT problem via enumerative combinatorics and the use of generating
functions.  
In a CSP formula where each clause relates to
exactly $2$ variables (or $2$-CSP),
it is very natural to represent the whole formula, i.e. the collection of clauses,
 via a graph where the clauses (resp. the variables)
are the edges (resp. the nodes) of the graph.

As far as graphs are concerned, the works of
Daud\'e and Ravelomanana~\cite{herve2011random} and Pittel and Yeum~\cite{YeumPittel} 
show that analyses based on graphical enumeration~\cite{Wright77,Wright80,HararyPalmer}
play a substantial role as such an approach  lead to
remarkably accurate results on the probability of satisfaction of a random formula. 
In the aforementioned works, the
enumerative approach consists in encoding the objects to be enumerated
with the help of generating functions.
%
Over the past decade
(see \cite{herve2011random,de2019symbolic,AofA20,DigraphWindow}
for reference)
it has turned out that it is fundamental to understand how to enumerate 
directed graphs before tackling the enumeration of $2$-SAT formulae.


    \subsection{Structure and phase transition}

\paragraph{Graphs.}
The two most natural models of random graphs
are $G(n,m)$ and $G(n,p)$.
Both generate graphs on $n$ vertices.
$G(n,m)$ also fixes the number $m$ of edges
and samples the graph uniformly at random,
while $G(n,p)$ adds an edge between each pair of vertices
independently with probability $p$.
Random graphs with a large number of vertices
exhibit similar limit statistical properties in both models,
as shown by Bollob\'as~\cite{bollobasrandomgraphs}.

Erd\H{o}s and R\'enyi~\cite{ER60} studied
the structure of random $G(n,m)$ graphs.
Let us transpose in the $G(n,p)$ model
one of their most striking results.
For $p = c/n$, with probability tending to $1$,
a random $G(n,p)$ graph contains, as \( n \to \infty \),
\begin{itemize}
\item
only trees and unicycles (components having only one cycle) if $c < 1$,
\item
only trees, unicycles, and a unique ``\emph{giant}'' component,
        of size $n \cdot f(c) + o(n)$, if $c > 1$.
\end{itemize}
Such a drastic change of behavior is called
a \emph{phase transition}
(a term borrowed from theoretical physics).
The typical structure of random $G(n,p)$ graphs
for $p$ close to $1/n$ was investigated
by Stepanov~\cite{stepanov1988some}
and described in fine detail
by Janson, Knuth, {\L}uczak and Pittel~\cite{janson1993birth}.
They showed that in the \emph{critical window},
corresponding to $p = \frac{1}{n}(1 + \bigO(n^{-1/3}))$,
a random $G(n,p)$ graph has a positive limit probability
to contain several connected components
that are neither trees nor unicycles.
They also derived the limit distributions
for various statistics of those components.
As we discuss below, the satisfiability
of random $2$-SAT formulae 
undergoes a similar phase transition.
Finally, an even more precise description of large graphs
in the critical window has been obtained
by Addario-Berry, Broutin and Goldschmidt~\cite{addario2012continuum},
in the form of a \emph{scaling limit}
(a geometrical limit for random graphs with a large number of vertices
expressed in terms of various stochastic processes related to the
Brownian motion).


\paragraph{Digraphs.}
The random models $G(n,m)$ and $G(n,p)$
extend naturally to digraphs.
In the $D(n,p)$ model,
the generated digraph has $n$ vertices and
each of the \( n(n-1) \) ordered pairs of distinct vertices
becomes an arc independently with probability $p$.
As in the case of graph,
the structure of digraphs undergoes a phase transition,
located by Karp~\cite{Karp1990} and \L{}uczak~\cite{Luczak1990}.
With probability tending to $1$ as $n$ tends to infinity,
the strongly connected components of a random $D(n,c/n)$ digraph
are
\begin{itemize}
\item
cycles or single vertices if $c < 1$,
\item
cycles, single vertices,
and a unique \emph{giant} strong component
containing a linear proportion of all vertices
if $c > 1$.
\end{itemize}
\L{}uczak and Seierstad~\cite{luczak2009critical}
derived the width $n^{-1/3}$ of the critical window.
For $p = \frac{1}{n}(1 + \bigO(n^{-1/3}))$,
they established that the size
of the largest strongly connected component is of order $n^{1/3}$.
Recently, Goldschmidt and Stephenson~\cite{Christina2019}
derived a scaling limit for random digraphs
inside the critical window.
Using a generating function approach,
Dovgal, de Panafieu, Ralaivaosaona, Rasendrahasina
and Wagner~\cite{DigraphWindow}
also obtained precise information on the typical structure
of random $D(n,p)$ digraphs inside the critical window.

\paragraph{2-SAT.}
In the scope of the current paper, we assume that inside of all the clauses of a
2-CNF, the literals have distinct Boolean variables. This results in \(
2n(n-1) \) possible clauses in case of \( n \) variables.
As for graphs, there are two natural models for random 2-CNF.
In the $(n,m)$ model, the numbers of variables $n$
and clauses $m$ are fixed,
and the formula is sampled uniformly at random.
This is akin to the $G(n,m)$ random graph model.
In the $(n,p)$ model, the number of variables $n$
and a probability $p$ are fixed,
and the formula is built by adding each of the \( 2n(n-1) \) possible
clauses independently with probability $p$.
This second model is akin to the $G(n,p)$ random graph model.
As in the case of graphs, we expect random formulae
to behave similarly under both models
as $n$ tends to infinity.
We refer to \cite{bollobas2001scaling} for details regarding the
correspondence between these models as \( n \to \infty \).

Our enumerative results allow us to express exactly
the probability for a random formula to be satisfiable
in the $(n,m)$ model.
In \cref{proposition:proba:np},
we show how to translate this result to the $(n,p)$ model as well.
The most popular model for random 2-SAT formulae is the $(n,p)$ model,
so we focus on it in the rest of the discussion.

Let $\proba_{\SAT}(n,p)$ denote the probability
for a random $(n,p)$ CNF to be satisfiable.
It was shown by Goerdt~\cite{goerdt92},
Chv\'atal and Reed~\cite{chvatalreed92}
and Fernandez de la Vega~\cite{de2001random}
that the limit of $\proba_{\SAT}(n,p)$ is $1$
for $p = c/n$ with $c < 1/2$, and $0$ if $c > 1/2$.
This sharp change is called the phase transition of $2$-SAT.
Bollob{\'a}s, Borgs, Chayes, Kim and Wilson~\cite{bollobas2001scaling}
refined their predictions and showed that
the limit probability of satisfiability $\proba_{\SAT}(n,p)$
is \( 1 - \Theta(|\mu|^{-3}) \)
(or, respectively, \( \exp(-\Theta(\mu^3)) \)) for
$p = \frac{1}{2 n} \left(1 + \mu n^{-1/3} \right)$
when \( \mu \to -\infty \) (or, respectively, \( \mu \to + \infty \)),
suggesting that the only region where this probability could be
non-trivial is for $\mu$ staying in a compact real interval.
Some of
these estimated were further refined by Kim~\cite{kim2008finding}
through Poisson cloning and Dovgal~\cite{dovgal2019birth}
with a different technique.
Further results on MAX SAT around the phase transition window
have been obtained in~\cite{coppersmith2004random}.

To describe the behavior of this phase transition,
we say that the critical window has width $n^{-1/3}$.
Let use define $\proba_{\SAT, \infty}(\mu)$ as the limit
\[
    \proba_{\SAT, \infty}(\mu) =
    \lim_{n \to +\infty}
    \proba_{\SAT} \left(
        n,
        \frac{1}{2 n} \left(1 + \mu n^{-1/3} \right)
    \right).
\]
A question of interest in the study of phase transitions
is the computation of the function $\proba_{\SAT, \infty}(\mu)$.
In the current paper,
we obtain exact expressions for the number of satisfiable 2-CNF
that are related to, though more complex than,
the expressions counting digraph families.
Our hope is that the analytic tools developed
by~\cite{DigraphWindow} to analyze the phase transition of digraphs
can be extended to 2-CNF and express $\proba_{\SAT, \infty}(\mu)$.
We also use those exact expressions to obtain
highly accurate (although non-rigorous) numerical predictions
of the curve of $\proba_{\SAT, \infty}(\mu)$.

    \subsection{Discussion of possible strategies}\label{sec:proof_strategies}

In the current work, we focus on the exact enumeration
of families related to $2$-SAT formulae.
To do this, we further extend the symbolic method
for enumeration of directed graphs.
Indeed, digraphs are closely related to $2$-SAT formulae
since the latter can be represented using implication digraphs.
An unsatisfiable formula is distinguished
by the presence of a specific subgraph -- a contradictory circuit -- inside its implication digraph.
We refer to an approach of Collet, de Panafieu, Gardy, Gittenberger and
Ravelomanana~\cite{collet2018threshold}
based on finding induced subgraphs in random graphs,
using generating functions as well.
Unfortunately, such an approach does not allow to capture
the counting recurrence in an efficient way,
because the family of the required patterns is too large.

Another possible approach is an inclusion-exclusion method, or a more
refined probabilistic tool based on distinguishing a random variable
inside a formula, \ie \emph{statistic}.
Two particular statistics can be used to count 2-SAT formulae:
the number of Boolean assignments satisfying the formula, and
the number of contradictory variables (defined below).
A 2-SAT formula is satisfiable if and only if
it has at least one satisfiable assignment.
Equivalently, it is satisfiable if and only if
it contains no contradictory variable.
%
Other statistics have historically been used to produce upper and
lower bounds on the probability that a random formula is satisfiable.

The expression obtained by applying inclusion-exclusion
on the number of satisfiable Boolean assignments indeed gives a
computable expression, which, however,
produces an exponential number of summands.
This approach is not promising, neither from a computational
nor from a theoretical viewpoint,
due to the rapid growth of magnitude of the alternating terms.
On the other hand, the first moment method applied to the number of satisfiable Boolean assignments
provides bounds on the location of the phase transition,
while the second moment method requires a more delicate choice of the
underlying random variable and fails to provide
a tight bound with this statistic for general $k$-SAT.
Recent breakthroughs in the asymptotic threshold of $k$-SAT
for large $k$~\cite{DingSlySun,coja2016asymptotic}
also rely on a careful choice of the right statistic:
the authors consider \emph{clusters} of solutions
instead of the total number of satisfying assignments.

Let us briefly turn to the work of Bollob\'as, Borgs, Chayes, Kim and
Wilson~\cite{bollobas2001scaling}.
Let us write $H \subset F$ if the clauses of the CNF $H$
are included in the clauses of the CNF $F$.
In the combinatorial study of random 2-SAT, and more generally,
$k$-SAT, one of the key features of a random formula has been its
\emph{spine}, which is defined as
\[
S(F) = \{ x \, | \, \exists H \subset F, \, H
\text{ is } \SAT \text{ and } H \land x \text{ is } \UNSAT \} \, .
\]
The spine can be seen as a set of literals that are forced to take
False values in any satisfying assignment.
However, for unsatisfiable formulae the spine is also 
well-defined.
As an alternative to this purely logical definition,
a spine in the implication digraph can be defined
as the set of literals $x$ for which there exists
a directed path from $x$ to $\n x$.
In other terms, the spine of a formula $F$ is defined 
as the set of literals $x$ such that there exists a satisfiable
subformula
$H$ of $F$ with the property that $H$ is SAT but $H \land x$ is not satisfiable.
Consequently, when building a formula by adding random clauses
one after the other starting from the empty formula,
the spine has proven to be a useful concept for calculating
the probability of satisfiability of the formula.
Introducing this concept allowed the authors of~\cite{bollobas2001scaling}
to establish the width $n^{-1/3}$ of the critical window
of the phase transition in 2-SAT,
\ie to prove that the limit of the probability $\proba_{\SAT}(n,p)$
could only be non-trivial for $p = \frac{1}{2n}(1 + \bigO(n^{-1/3}))$.

In the current paper, we are considering the so-called
contradictory strongly connected components as the central parameter.
In an implication digraph, a variable \( x \) is contradictory if
there is a path from \( x \) to \( \n x \) and from \( \n x \) to \(
x\), and the whole strongly connected component containing a
contradictory variable is called contradictory (we will show that, in
fact, every variable of this component will also be contradictory).
On the level of logical definition,
a Boolean variable $x$ is called contradictory
if both its literals $x$ and $\n x$
belong to the spine. Equivalently, the set of contradictory variables
is defined as
\[
    C(F) = \{
        x \, | \, \exists H_1, H_2 \subset F, \, H_1, H_2 \text{ are $\SAT$ and }
        H_1 \land x \text{ and } H_2 \land \n x \text{ are $\UNSAT$}
    \}.
\]
This definition can be possibly extended to CSP models other than 2-SAT.

\subsection{Our results}
\label{sec:our:results}

\paragraph{Exact enumeration.}
In the current paper, we express the number of satisfiable 2-CNF formulae
with the help of generating functions.
It follows from~\cref{th:cnf:scc:cscc} that if \( a_{n,m} \) denotes the number
of satisfiable 2-CNF with $n$ Boolean variables and $m$ clauses, then
\( a_{n,m} \) can be encoded into a generating function identified by
the expression
\[
    \sum_{n \geq 0} \dfrac{1}{(1+w)^{n(n-1)}} \dfrac{z^n}{2^n n!}
    \sum_{m \geq 0} a_{n,m} w^m
    =
    \dfrac
    {
        \displaystyle\sum\limits_{n \geq 0}
        \frac{1}{(1+w)^{n(n-1)}} \frac{z^n}{2^n n!}
        \odot_z
        e^{-\SCC(2z, w)/2}
    }{
        \displaystyle\sum\limits_{n \geq 0}
        \frac{1}{(1+w)^{{n \choose 2}}} \frac{z^n}{n!}
        \odot_z
        e^{-\SCC(z, w)}
    }\, ,
\]
where \( \odot_z \) denotes the exponential Hadamard product
\[
    \sum_{n \geq 0} a_n(w) \dfrac{z^n}{n!}
    \odot_z
    \sum_{n \geq 0} b_n(w) \dfrac{z^n}{n!}
    :=
    \sum_{n \geq 0} a_n(w) b_n(w) \dfrac{z^n}{n!}\, , 
\]
and \( \SCC(z, w) \) is the Exponential generating function of strongly connected digraphs (see ~\cite[Corollary 3.5]{de2019symbolic} or \cref{eq:scc} below).
Note that negative signs in formal generating functions
can often be interpreted as an application of the inclusion-exclusion
principle, as explained by \cite[Lemma 2.2.29]{goulden2004combinatorial}
(see also \cite[III. 7.4, p.~206]{flajolet2009analytic}).
In our case, the inclusion-exclusion manifests itself in the negative exponential terms, and
the corresponding statistical patterns inside a random implication digraph are the number
of the so-called contradictory and ordinary strongly connected components.

We obtain two exact expressions for the number of satisfiable 2-SAT formulae
(\cref{th:satisfiable} and \cref{th:variant:satisfiable}),
the number of unsatisfiable 2-SAT formulae
whose implication digraph is strongly connected
(\cref{th:CSCC})
as well as a description of the structure of the implication digraphs
associated to 2-SAT formulae (\cref{th:detailed:CNF}): the latter result describes
the implication digraphs with given allowed strongly connected components.
Those results are based on generating function manipulations.
In this paper, we introduce a new type of generating function,
called an \emph{Implication generating function}.
It is inspired by the \emph{Special} or \emph{Graphic} generating function
introduced in \cite{Robinson71,Gessel95}.
The product of an Implication generating function
with a Graphic generating function corresponds to a combinatorial operation
involving a digraph and a 2-SAT formula,
that we call hereafter an \emph{implication product}.

\paragraph{Phase transition.}
Note that Deroulers and Monasson~\cite{Deroulers} gave numerical estimates of probabilities of these formula being satisfiable around their phase transition.
Up to $n=5 \times 10^6$,
they were able to determine the empirical values of
\[
\proba\left[ \mbox{Random $2$-SAT formula built with $n$ variables and $n$ clauses is SAT} \right]
\]
using Monte-Carlo simulation, and gave a prediction of
\(
    \proba_{\SAT, \infty}(0) = 0.907 \pm 10^{-3}.
\)
Our results translate into more efficient algorithms
to accurately but non-rigorously predict those empirical values.
We improve their prediction to
\(
    \proba_{\SAT, \infty}(0) = 0.90622396067 \pm 10^{-11},
\)
and also make prediction
for $\proba_{\SAT, \infty}(\mu)$ for other values of $\mu$,
plotting the phase transition curve of the 2-SAT inside the critical window, viz.~\cref{fig:limiting:plot}.
For the sake of reproducibility, we share the \texttt{ipython} notebooks used for
computation of these values by a hopefully permanent public URL on \texttt{GitLab}
\begin{center}
\url{https://gitlab.com/sergey-dovgal/enumeration-2sat-aux}
\end{center}
Other researchers are encouraged to reuse them if they wish so.
Given the similarity between the exact expressions counting
digraph families and satisfiable 2-CNF,
we hope that the analytic tools for analysis of the phase transition of digraphs
similar to the ones developed by~\cite{DigraphWindow} 
can be extended to 2-CNF and express $\proba_{\SAT, \infty}(\mu)$
in a closed form.
We expect this curve to have an expression akin to
the integrals of Airy functions such as those encountered
in~\cite{DigraphWindow}, or some form of generalized Airy function
(see~\cite{janson1993birth,de2015phase}).

\begin{figure}[hbt!]
    \begin{center}
        \includegraphics[width=0.7\textwidth]{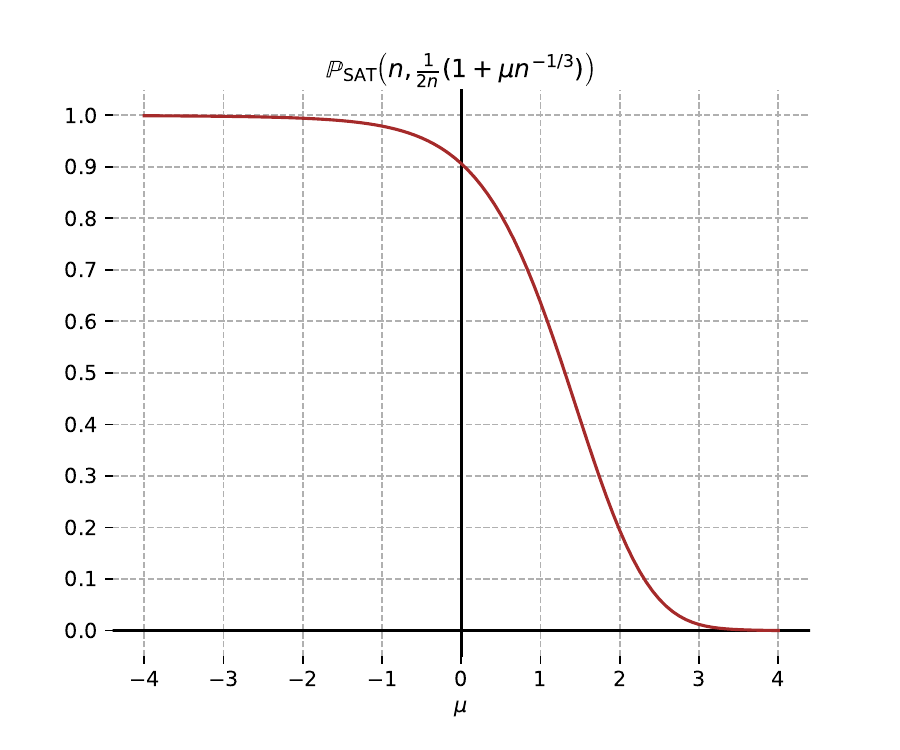}
    \end{center}
    \caption{\label{fig:limiting:plot} The predicted limiting probability that a formula is satisfiable inside the critical window of the 2-SAT phase transition, in the model
    \( (n, p) \), where $n$ denotes the number of Boolean variables, and $p$ is the clause
    probability, $p = \frac{1}{2n}(1 + \mu n^{-1/3})$.}
\end{figure}

    \subsection*{Outline of the paper}

We recall the classic definitions of 2-SAT formulae
as well as the characterization of satisfiable formulae
and the structure of the associated implication digraphs
in \cref{sec:conjunctive:normal:forms}.
The various types of generating functions used throughout this article
are introduced in \cref{sec:generating:functions}.
Then, \cref{sec:counting:sat:families}
presents our results and their proofs.
Finally, in~\cref{sec:numerical} we provide the first several terms of
the counting sequences for some 2-SAT families along with accurate
numerical predictions related to the satisfiability phase transition.

\section{Conjunctive Normal Forms and implication digraphs}
\label{sec:conjunctive:normal:forms}

\subsection{Definitions and notation} \label{sec:definitions:notations}
In this section, we are using the classical binary Boolean operators $\lor,
\land$ and $\to$ which correspond respectively to \emph{disjunction},
\emph{conjunction} and \emph{implication}.
We are also using the unary operator $\neg x$ or $\n x$ to denote negation.

\begin{definition}
The \emph{literals} of a Boolean variable $x$ are $x$ and its negation $\n x$.
A \emph{Conjunctive Normal Form} (CNF) formula
on $n$ variables is a set (conjunction) of clauses,
where each clause is a disjunction of literals
corresponding to distinct variables.
\end{definition}

For example, the formula (viz.~\cref{fig:cnf:example})
\begin{equation} \label{eq:two:sat:example}
    (x_1 \vee x_3) \land (\n x_2 \vee x_1) \land (x_2 \vee \n x_3)
\end{equation}
is considered to be the same CNF as
\[
    (x_1 \vee \n x_2) \land (x_2 \vee \n x_3) \land (x_3 \vee x_1) 
    .
\]
Throughout this work, we do not consider formulae with literal or clause duplications, such as
\begin{align*}
    (x_1 \vee \n x_1) &\land (x_2 \vee x_1),\\
    (x_1 \vee \n x_2) &\land (\n x_2 \vee x_1).
\end{align*}

\begin{definition}
A CNF is \emph{satisfiable} if there exists
an assignment of Boolean values to the variables
that satisfies each clause.
A \emph{2-CNF} (or \emph{2-SAT formula})
is a CNF where each clause contains exactly two (distinct) literals.
\end{definition}

We also define a second type of Boolean formulae,
as an intermediate step between $2$-CNF and directed graphs (digraphs).

\begin{definition}
An \emph{implication formula} on $n$ variables is a set of clauses,
where each clause is the implication of two literals
corresponding to distinct variables.
\end{definition}

\begin{figure}[hbt]
\begin{minipage}[t]{.48\textwidth}
    \begin{center}
        \begin{tikzpicture}[>=stealth',thick, node distance=1.5cm] 
  \hpoint{1}{1}{.25}{+0.5};
  \point{2}{-1} {1}{-0.5};
  \point{3}{-1}{-1}{-0.5};
  \path (1a) edge [bblue,very thick] (2b);
  \path (2a) edge [bblue,very thick] (3b);
  \path (3a) edge [bblue,very thick] (1a);
\end{tikzpicture}
    \end{center}
    \caption{\label{fig:cnf:example} An example of a 2-CNF depicted in a
    form of a labeled graph with decorations.
    Literals \( x_k \) and \( \n x_k \) are replaced with \( k \)
    and \( \n k \) for brevity.
    }
\end{minipage}\hfill%
\begin{minipage}[t]{.48\textwidth}
    \begin{center}
        \begin{tikzpicture}[>=stealth',thick, node distance=1.0cm] 
\draw
node[arn_n](1) at (0, 0) {\( 1\)}
node[arn_n](2) at (0, -1) {\( 2 \)}
node[arn_n](3) at (0, -2) {\( 3 \)}
node[arn_r](1') at (2, 0) {\( \fn 1 \)}
node[arn_r](2') at (2, -1) {\( \fn 2 \)}
node[arn_r](3') at (2, -2) {\( \fn 3 \)}
;
\fromto{1'}{2'}[][][-30][-10]
\fromto{2}{1}[][][-30][-10]
\fromto{2'}{3'}[][][-30][-10]
\fromto{3}{2}[][][-30][-10]
\fromto{3'}{1}[][][4]
\fromto{1'}{3}[][][4]
\end{tikzpicture}
    \end{center}
    \caption{\label{fig:implication:example} The corresponding implication digraph.
    Literals \( x_k \) and \( \n x_k \) are replaced with \( k \) and
    \( \n k \) for brevity.
    }
\end{minipage}
\end{figure}

Since the clause $x \vee y$ is equivalent with either of the clauses
\( (\n x \to y) \), \( (\n y \to x) \), and with their conjunction
$(\n x \to y) \land (\n y \to x)$,
any 2-CNF has an equivalent implication formula.
Reciprocally, any implication formula where each clause $x \to y$
has its symmetric $\n y \to \n x$ also
corresponds to a valid 2-CNF.
For example (viz.~\cref{fig:implication:example}),
the implication formula corresponding to \eqref{eq:two:sat:example} is
\[
    (\n x_1 \to x_3) \land (\n x_3 \to x_1) \land
    (x_2 \to x_1) \land (\n x_1 \to \n x_2) \land
    (\n x_2 \to \n x_3) \land (x_3 \to x_2).
\]
Therefore, a 2-CNF is satisfiable if and only if
the corresponding implication formula is satisfiable.

To each 2-CNF on $n$ variables with $m$ clauses, there corresponds
an \emph{implication digraph} with $2n$ vertices and $2m$ arcs:
each clause $(k, \ell)$ corresponds to the arcs
$(\n k, \ell)$ and $(\n \ell, k)$.
Since a clause cannot contain twice the same variable,
an implication digraph contains neither loops
nor arcs from a literal to its negation.
Since a 2-CNF does not contain twice the same clause,
there is at most one arc between any two literals.

\begin{definition} \label{def:SCCs}
A \emph{contradictory variable} in a 2-CNF
is a variable $x$ such that the implication digraph contains
oriented paths from $x$ to $\n x$ and from $\n x$ to $x$.
A \emph{strongly connected component} (SCC) of a digraph
is a set of vertices, maximal for the inclusion with respect to the property
that an oriented path exists between any two vertices from the set.
In an implication digraph, a \emph{contradictory strongly connected component}
(contradictory SCC) is an SCC that contains a contradictory variable.
An SCC that is not contradictory is \emph{ordinary}.
\end{definition}

\begin{definition}
An SCC of a digraph is \emph{source-like}
if there is no arc pointing to any of its vertices
from a vertex outside of it.
It is \emph{sink-like}
if there is  no arc pointing from any of its vertices
to a vertex outside of it.
It is \emph{isolated} if it is both source-like and sink-like.
\end{definition}

\begin{definition}
Consider a digraph \( D \) whose vertices form a subset of literals
\( \{x_1, \ldots, x_n, \n x_1, \ldots, \n x_n\} \).
The \emph{negation} \( \n D \) of \( D \) is formed
by replacing its vertex labels with their negations
and flipping the edge directions:
if the original digraph \( D \) contains an arc \( x \to y \),
then its negated digraph contains an arc \( \n y \to \n x \) instead.
\end{definition}

\begin{figure}[hbt!]
\begin{minipage}[t]{.4\textwidth}
\begin{center}
    \raisebox{.3cm}{\scalebox{1}{\begin{tikzpicture}[>=stealth',thick, node distance=1.0cm]
    \draw
    node[arnBleuGrande](a) at (0,-2) {$A$}
    node[arnBleuGrande](a') at (0,0) {$\n A$}
    node[arnOrangeGrande](xx') at (1,-1) {$X$}
    node[arnOrangeGrande](yy') at (4,-1) {$Y$}
    node[arnVertGrande](b) at (2.5,0) {$B$}
    node[arnVertGrande](b') at (2.5,-2) {$\n B$}
    node[arnVioletGrande](c) at (5,0) {$C$}
    node[arnVioletGrande](c') at (5,-2) {$\n C$}
    ;
    \aprod{a}{xx'};
    \aprod{xx'}{b};
    \aprod{yy'}{b};
    \aprod{b'}{xx'};
    \aprod{b'}{yy'};
    \aprod{xx'}{a'};
    \path (a) edge [blackred,thick, bend left=35,
    decoration={markings,mark=at position 0.99 with
    {\arrow[ultra thick,blackred, rotate=0]{>}}}, postaction={decorate}
    ] node {} (a');
\end{tikzpicture}}}

    \caption{\label{fig:condensation}A condensation of 2-CNF implication digraph. Each component is depicted as a node. An arrow from component $D$ to component $E$ means that there is at least one edge $u \rightarrow v$ in the implication digraph with the literal $u$ in $D$ and the literal $v$ in $E$.}
\end{center}
\end{minipage}\hfill%
\begin{minipage}[t]{.25\textwidth}
\begin{center}
    \raisebox{.05cm}{\scalebox{1.1}{\begin{tikzpicture}[>=stealth',thick, node distance=1.0cm] 
\draw
node[arn_n](1) at (0, 0) {\( 1\)}
node[arn_r](2') at (1.41, 1)  {\( \fn 2 \)}
node[arn_n](3) at (1.41, -.8)  {\( 3 \)}
;
\fromto{1}{2'}[][][-4]
\fromto{2'}{3}[][][-4]
\fromto{3}{1}[][][-4]
\node[rectangle,dashed,draw,fit=(1)(2')(3), very thick,
      rounded corners=5mm,inner sep= 5pt, bgreen] {};
\end{tikzpicture}}}

    \caption{\label{fig:oscc}An ordinary component.}
\end{center}
\end{minipage}\hfill%
\begin{minipage}[t]{.25\textwidth}
\begin{center}
    \scalebox{0.7}{\begin{tikzpicture}[>=stealth',thick, node distance=1.0cm] 
\draw
node[arn_n](1)  {\( 1\)}
node[arn_n](2)  [below of=1] {\( 2 \)}
node[arn_n](3)  [below of=2, node distance=.85cm] {\( 3 \)}
node[arn_n](4)  [below of=3] {\( 4 \)}
node[arn_r](1') [right of=1, node distance=2cm] {\( \fn 1 \)}
node[arn_r](2') [right of=2, node distance=2cm] {\( \fn 2 \)}
node[arn_r](3') [right of=3, node distance=2cm] {\( \fn 3 \)}
node[arn_r](4') [right of=4, node distance=2cm] {\( \fn 4 \)}
;
\fromto{1}{2'}[][][4]
\fromto{2}{1'}[][][4]
\fromto{3'}{4}[][][4]
\fromto{4'}{3}[][][4]
\fromto{1}{2}[][][ 30][10]
\fromto{2'}{1'}[][][ 30][10]
\fromto{1'}{3'}[][][-50][-5]
\fromto{3}{1}[][][-50][-5]
\fromto{4}{3}[][][-60][-15]
\fromto{3'}{4'}[][][-60][-15]
\node[rectangle,dashed,draw,fit=(1)(1')(4)(4'), very thick,
      rounded corners=5mm,inner sep= 13pt, bred] {};
\end{tikzpicture}}

    \caption{\label{fig:cscc}A contradictory component.}
\end{center}
\end{minipage}
\end{figure}

In~\cref{fig:condensation} we provide an example,
where an implication digraph of a 2-CNF formula
is depicted in a condensated form, the components \( X = \n X \)
and \( Y = \n Y \) are contradictory SCCs;
\( A \), \( \n B \), \( C \) and \( \n C \)
are \emph{ordinary} source-like SCCs;
\( \n A \), \( B \), \( C \) and \( \n C \)
are \emph{ordinary} sink-like SCCs; and finally,
\( C \) and \( \n C \)
are \emph{ordinary isolated} SCCs.
\Cref{fig:oscc,fig:cscc} provide examples of an ordinary and a
contradictory component.

\subsection{Structural properties of 2-Conjunctive Normal Forms}
\label{sec:structural:properties:cnf}

The first linear time algorithm to decide
the satisfiability of a 2-SAT formula
was designed by Aspvall, Plass and Tarjan
\cite{APT79}.
It relied on a characterization
of the implication digraphs
of satisfiable 2-SAT formulae.
In this section, we recall their proof,
reformulating it to fit our needs.

We start with properties of the contradictory SCCs of implication digraphs.

\begin{proposition} \label{th:structure:CSCC}
Let \( C \) be a contradictory SCC. Then,
\begin{enumerate}
    \item All variables appearing in \( C \) are contradictory;
    \item If \( C \) is source-like (or sink-like) then it is isolated;
    \item If \( C' \) is another contradictory SCC, then there is no oriented path starting in \( C \) and ending in \( C' \).
\end{enumerate}
\end{proposition}

\begin{proof}
  If $C$ is reduced to one variable, the assertion is trivial.
  So, let $C$ be a contradictory SCC that contains a
 contradictory variable $x$ and a literal $k$.
Then $C$ contains oriented paths
from $x$ to $\n x$,
from $\n x$ to $x$,
from $x$ to $k$ and
from $k$ to $x$.
By symmetry of the arcs,
this implies the existence of oriented paths from $\n k$ to $\n x$
and $\n x$ to $\n k$.
Combining them, we obtain oriented paths from $k$ to $\n k$ and $\n k$ to $k$.
Thus, $k$ is also a contradictory variable.

For the second part, suppose that \( C \) is source-like.
If there were an arc from a literal $k \in C$
to a literal $\ell$ outside of $C$,
then by symmetry we would also have the arc $(\n \ell, \n k)$.
By the previous result, $\n \ell$ does not belong to $C$, but $\n k$ does.
Thus, $C$ would not be source-like, which leads to a contradiction.

Let us prove the last part of the proposition.
By symmetry, an oriented path from a literal $k$
of the contradictory SCC $C$
to a literal $\ell$ of the contradictory SCC $C'$
implies an oriented path from $\n \ell$ to $\n k$,
hence from $C'$ to $C$.
Thus, $C$ and $C'$ are the same contradictory SCC.
\end{proof}

We now turn to ordinary source-like SCC.

\begin{proposition} \label{th:ordinary:SCC:symmetric}
Let $C$ denote a source-like (\resp sink-like, \resp isolated)
ordinary SCC from the implication digraph,
then the component $\n C$
is a sink-like (\resp source-like, \resp isolated) SCC which is disjoint from $C$.
\end{proposition}

\begin{proof}
It is sufficient to prove that if $C$ is ordinary and source-like,
then $\n C$ is a sink-like SCC.
Since each arc $(k, \ell)$ has its symmetric $(\n \ell, \n k)$,
any oriented path from $k$ to $\ell$ has as symmetric
an oriented path from $\n \ell$ to $\n k$.
Thus, the negation of the literals from $C$ also form an SCC $\n C$.
It is disjoint from $C$ because $C$ is not contradictory.
Furthermore, any arc pointing from $\n C$ to a literal $k$ outside $\n C$
would have its symmetric pointing to $C$ from $\n k$.
The literal $\n k$ could not belong to $C$, otherwise $k$ would be in $\n C$.
Thus, $C$ would not be source-like: a contradiction.
It follows that such a $k$ cannot exist and that $\n C$ is sink-like.
\end{proof}

We finally arrive at the classic characterization
of satisfiable 2-CNF.

\begin{proposition}
A 2-CNF is satisfiable if and only if
it contains no contradictory variable.
\end{proposition}

\begin{proof}
Given an implication digraph, let us write $x \rightsquigarrow y$
if there exists an oriented path
from the literal $x$ to the literal $y$.
Suppose that a formula contains a contradictory variable and is
satisfiable, which means that all variables can be assigned Boolean
values satisfying all the clauses.
If $x$ is a contradictory variable, then in the implication digraph,
the presence of an oriented path $x \rightsquigarrow \n x$
results in the logical implication $x \Rightarrow \n x$,
so $x$ must take the value $\False$.
Similarly, the oriented path $\n x \rightsquigarrow x$ implies $\n x \Rightarrow x$, so $x$ must be $\True$.
This is a contradiction, so the formula cannot be satisfiable
if it contains a contradictory variable.

We prove the reverse by induction on the number of variables.
Consider a 2-CNF without contradictory variables,
and assume that any smaller 2-CNF without contradictory variables is satisfiable.
Then, it must contain a source-like SCC.
Let $C$ be any source-like SCC of the implication digraph.
By \cref{th:ordinary:SCC:symmetric}, its symmetric $\n C$ is sink-like.
Let us set all the literals of $C$ to the Boolean value $\False$,
so the literals from $\n C$ are set to $\True$.
Consider the formula $F$ corresponding to removing the variables from $C$ and $\n C$.
It contains no contradictory variable, so it is satisfiable by induction.
We claim that any solution to $F$ satisfies the original formula.
Indeed, the arcs we removed are from $C$,
in which case they become $\False \Rightarrow x$
which is satisfied for any value of $x$,
or to $\n C$, in which case they become $x \Rightarrow \True$
which is satisfied for any value of $x$.
\end{proof}

\section{Generating functions and implication product}
\label{sec:generating:functions}

This section introduces the tools
that will be applied in \cref{sec:counting:sat:families}
for counting various $2$-SAT families.
First, let us recall the general definition of generating functions
and give a brief overview of the plan of the section.

A generating function (GF) associated to a family $\mA$
is a formal series of the form
\[
    A(z) = \sum_{a \in \mA} \kappa_{|a|}(z),
\]
where $z$ is a formal variable,
$|a|$ is a non-negative integer number computed from $a$
and called the \emph{size} of \( a \),
and $(\kappa_n(z))_{n \geq 0}$ is a sequence of functions.
Let $a_n$ denote the number of elements $a \in \mA$
such that $|a| = n$, and assume $a_n$ is finite for all $n \geq 0$.
Grouping the summands corresponding to the same value $|a|$,
the GF becomes
\[
    A(z) = \sum_{n \geq 0} a_n \kappa_n(z).
\]
The \emph{type} of the GF corresponds to the choice of $(\kappa_n(z))$.
The founding idea of analytic number theory~\cite{tenenbaum2015introduction},
species theory~\cite{bergeron1998combinatorial}
and analytic combinatorics~\cite{flajolet2009analytic}
is that the study of the sequence $(a_n)$ can be simplified
by the introduction of the right type of GF.

In this article, we will use several types of GFs,
defined in \cref{sec:various:types}
\begin{itemize}
\item
the classic \emph{Ordinary} and \emph{Exponential} GFs
correspond to
\[
    \kappa_n(z) = z^n
    \qquad \text{and} \qquad
    \kappa_n(z) = z^n / n!,
\]
and a good reference is \cite[Chapters I and II]{flajolet2009analytic},
\item
\emph{Graphic} (or \emph{Special}) GFs,
introduced by Robinson~\cite{Robinson71} and Gessel~\cite{Gessel95},
correspond to
\[
    \kappa_n(z) =
    \frac{1}{(1+w)^{\binom{n}{2}}}
    \frac{z^n}{n!}
\]
where $w$ denotes an additional formal variable,
\item
and a new type of GFs, which we call \emph{Implication} GFs,
that corresponds to
\[
    \kappa_n(z) =
    \frac{1}{(1+w)^{n (n-1)}}
    \frac{z^n}{2^n n!} \, .
\]
\end{itemize}
The Implication GF is designed for the enumeration
of various $2$-CNF families.
To translate GFs of a type into another type,
we will use the \emph{exponential Hadamard product},
introduced by Joyal (see \cref{rem:hadamard:history}).
It is presented in \cref{sec:hadamard}.
The central idea behind the use of GFs
is that combinatorial operations on the families
translate into analytic operations on their GFs.
The type of GF used depends
on the combinatorial operation of interest.
\cref{sec:combinatorial:operations} presents
those operations for the various types used in this paper.

    \subsection{Types of generating functions} \label{sec:various:types}

We present the various types of GFs used in this paper
in the specific context of graphs, digraphs and implication digraphs.
General introductions to Ordinary and Exponential GFs
are available in \cite{bergeron1998combinatorial}
and \cite{flajolet2009analytic}.
The definition of Graphic (or \emph{Special}) GFs
is due to Robinson~\cite{Robinson71} and Gessel~\cite{Gessel95}.
The definition of Implication GFs is new
and will be motivated in \cref{sec:combinatorial:operations}.
In this paper we are dealing with labeled objects: for each graph or
digraph with \( n \) vertices, we consider that their vertices are
labeled with distinct labels from \( \{ 1, \ldots, n\} \)
(more on that also in~\cref{sec:combinatorial:operations}).

\begin{definition} \label{def:counting:digraphs}
Let $\mA$ be a graph or digraph family
and $\mB$ an implication digraph family.
Let $a_{n,m}$ denote the number of (di)graphs in $\mA$
with $n$ vertices and $m$ arcs (\resp edges),
and $b_{n,m}$ the number of implication digraphs in $\mB$
with $2n$ vertices and $2m$ implication arcs (\ie built from 2-CNFs
with $n$ variables and $m$ clauses).
For each $n$, let $\mA_n$ denote the subfamily of $\mA$
of (di)graphs on $n$ vertices,
and $\mB_n$ the subfamily of $\mB$
of implication digraphs on $2n$ vertices.

The \emph{Ordinary} GFs of $\mA_n$ and $\mB_n$ are defined as
\[
    a_n(w) = \sum_{m \geq 0} a_{n,m} w^m
    \qquad \text{and} \qquad
    b_n(w) = \sum_{m \geq 0} b_{n,m} w^m.
\]

The Exponential GF $A(z,w)$, Graphic GF $\widehat A(z,w)$
and Implication GF $\ddot B(z,w)$ are defined as
\[
    A(z,w) = \sum_{n \geq 0} a_n(w) \frac{z^n}{n!},
    \quad
    \widehat A(z,w) =
    \sum_{n \geq 0}
    \frac{a_n(w)}{(1+w)^{\binom{n}{2}}}
    \frac{z^n}{n!}
    \quad \text{and} \quad
    \ddot B(z,w) =
    \sum_{n \geq 0}
    \frac{b_n(w)}{(1+w)^{n(n-1)}}
    \frac{z^n}{2^n n!}.
\]
To alleviate the notations, we often omit the variable $w$,
writing $A(z)$ instead of $A(z,w)$.
As a convention, we will use hats to distinguish Graphic GFs
from Exponential GFs, and double dots to denote Implication GFs.
\end{definition}

The name \emph{Graphic} of the generating function
may be somewhat misleading:
we always use Exponential GFs to enumerate graphs,
and often use Graphic GFs to enumerate \emph{directed} graphs.
However, for historical reasons, we are keeping its original name.

The following lemma expresses the generating functions
of various families that will be used throughout this article.

\begin{lemma} \label{th:classic:gfs}
Let $G(z)$ (variable $w$ omitted)
denote the Exponential GF of all graphs with labeled vertices,
where loops and multiple edges are forbidden, then
\[
    G(z) =
    \sum_{n \geq 0}
    (1+w)^{\binom{n}{2}}
    \frac{z^n}{n!}.
\]
Let $D(z)$ denote the Exponential GF of all digraphs,
with labeled vertices,
where loops and multiple arcs are forbidden.
In this model, we assume that between any two nodes of a digraph,
both arcs connecting these nodes can be present. Then
\[
    D(z) =
    \sum_{n \geq 0}
    (1+w)^{n (n-1)}
    \frac{z^n}{n!}.
\]
Let $\gset(z)$ denote the Graphic GF of digraphs
that contain no arcs, then
\[
    \gset(z) =
    \sum_{n \geq 0}
    \frac{1}{(1+w)^{\binom{n}{2}}}
    \frac{z^n}{n!}.
\]
Let $\idset(z)$ denote the Implication GF of implication digraphs
that contain no arcs, then
\[
    \idset(z) =
    \sum_{n \geq 0}
    \frac{1}{(1+w)^{n(n-1)}}
    \frac{z^n}{2^n n!}.
\]
\end{lemma}

\begin{proof}
Consider a family $\mA$ of graphs
(\resp digraphs, \resp implication digraphs)
and let $a_{n,m}$ denote the number of its elements
with $n$ vertices (\resp $n$ vertices, \resp $2n$ vertices)
and $m$ edges (\resp $m$ arcs, \resp $2m$ arcs).
Let $a_n(w)$ denote its Ordinary GF
\[
    a_n(w) =
    \sum_{m \geq 0}
    a_{n,m} w^m.
\]
The number of graphs with $n$ vertices and $m$ edges is
$\binom{\binom{n}{2}}{m}$,
because a subset of $m$ edges is chosen
among all possible $\binom{n}{2}$ edges.
Thus, if $\mA$ is the family of all graphs,
then $a_n(w) = (1 + w)^{\binom{n}{2}}$
and the first result follows.
When $\mA$ is the family of all digraphs,
we have $a_{n,m} = \binom{n (n-1)}{m}$,
because a subset of $m$ arcs is chosen
among all possible $n (n-1)$ arcs.
Thus, $a_n(w) = (1+w)^{n (n-1)}$,
which implies the second result.
When $\mA$ is the family of digraphs without any arc,
we have $a_n(w) = 1$, because there exists only one digraph
on $n$ vertices containing no arc.
This implies the third result.
Similarly, when $\mA$ is the family of implication digraphs containing no arc,
we have again $a_n(w) = 1$.
This implies the fourth result.
\end{proof}

We choose the notations $\gset$ and $\idset$
to represent families that are just set of vertices,
without any additional structure.
Note that in all the sums of the last lemma,
$n=0$ corresponds to the empty graph, containing no vertices.

\paragraph{Probability from generating functions.}
It is handy to use generating functions to calculate the probabilities
in the $(n, p)$ model, where the number of clauses is not fixed,
but each clause is drawn independently with probability $p$
(\cf~\cite[Lemma 6]{AofA20} or \cite[Lemma 2.8]{DigraphWindow}).

\begin{proposition}
\label{proposition:proba:np}
Let \( \mF \) be some family of 2-SAT formulae,
whose Implication GF is \( \ddot{F}(z, w) \).
Then, the probability that a random formula from the \( (n,p) \)
model (\ie a random formula with $n$ Boolean variables where each of
the $2n(n-1)$ possible clauses is drawn
independently with probability $p$) belongs to $\mF$, is
\[
    \mathbb P_{n,p}(F \in \mF) = 2^n n! (1-p)^{n(n-1)}
    [z^n]\ddot{F}\left(z, \frac{p}{1-p}\right).
\]
\end{proposition}
\begin{proof}
Let \( m(F) \) denote the number of clauses of a formula \( F \),
and let \( \mF_n \) be the set of the formulae from \( \mF \)
containing \( n \) Boolean variables.
There are \( 2n(n-1) \) possible clauses
for \( n \) Boolean variables.
The probability that a random formula \( F \) belongs to \( \mF_n \)
is expressed by summing over all possible numbers of edges
\begin{align*}
\mathbb P_{n,p}(F \in \mF)
&=
\sum_{F \in \mathcal F_n}
p^{m(F)} (1-p)^{2n(n-1)-m(F)}\\
&=(1-p)^{2n(n-1)} \sum_{F \in \mathcal F_n} \left(
\dfrac{p}{1-p}
\right)^{m(F)}
\\ &=
(1-p)^{2n(n-1)} \left. (1+w)^{n(n-1)} \right|_{w=\frac{p}{1-p}}
2^n
n! [z^n]
\ddot{F}\left(
z, \dfrac{p}{1-p}
\right)
\\ &=
2^n n! (1-p)^{n(n-1)} [z^n]
\ddot{F}\left(
z, \dfrac{p}{1-p}
\right).
\end{align*}
\end{proof}

\paragraph{Multivariate generating functions.}
For graph-like families, we use GFs with two variables:
$z$ marking the number of vertices
and $w$ the number of edges.
Additional marking variables can be introduced.
There are two ways to define the generating function
with several variables.
One way is to consider generalized counting sequences
(such as \( a_{n,k,j} \) in the case of three parameters),
and take the sum over all possible combination of indices
(\eg \( A(z, w, u) := \sum_{n,k,j \geq 0} a_{n,k,j} u^j w^k \frac{z^n}{n!} \)).
Another viewpoint is to say that the objects
inside the family do not have the same weight,
and an object receives a weight \( u^j \)
if the corresponding parameter marked by the variable \( u\)
inside this object is equal to \( j \).
In this case we return to the usual counting sequence
with one variable
\(
    \sum_{n \geq 0} a_n \frac{z^n}{n!}
\),
where \( a_n \) now denotes the total weight
of the objects of size \( n \),
which now depends on the additional marking variables.

    \subsection{Exponential Hadamard product} \label{sec:hadamard}

The \emph{exponential Hadamard product} $\odot_z$
of two formal power series \emph{with respect to the variable $z$},
is defined as
\[
    \bigg( \sum_{n \geq 0} a_n(w) \frac{z^n}{n!} \bigg)
    \odot_z
    \bigg( \sum_{n \geq 0} b_n(w) \frac{z^n}{n!} \bigg)
    =
    \sum_{n \geq 0} a_n(w) b_n(w) \frac{z^n}{n!}.
\]
In the following, all Hadamard products
are taken with respect to the variable $z$,
so we will omit its mention in the notation,
writing $\odot$ for $\odot_z$.

\begin{remark} \label{rem:hadamard:history}
The exponential Hadamard product was introduced in
\cite[Section 2.1, p. 64]{bergeron1998combinatorial}
as a \emph{Cartesian product} or just the \emph{Hadamard product}.
This notion (``\emph{cet ami oublié}'')
can be traced to the 1981 paper of Joyal
\cite[Theorem 3, equation (8)]{joyal1981combinatorial},
although Joyal does not really provide a proper definition.
To avoid confusion with the apparently more well-known
ordinary Hadamard product, we keep the word ``exponential''.
\end{remark}

The exponential Hadamard product satisfies
the following two elementary properties.
For any value $\alpha$ and series $A(z)$ and $B(z)$,
we have
\[
    A(\alpha\, z) \odot B(z) =
    A(z) \odot B(\alpha\, z)
    \qquad \text{and} \qquad
    A(z) \odot e^z = A(z).
\]

\begin{proposition}
\label{proposition:conversion}
The Graphic GF and Implication GF of sets (digraphs without arcs),
given in \cref{th:classic:gfs},
are used to convert the Exponential GF $A(z,w)$
into a Graphic GF $\widehat{A}(z,w)$
or an Implication GF $\ddot{A}(z,w)$
using the exponential Hadamard product with respect to $z$ as follows:
\[
    \widehat{A}(z,w)
    =
    A(z,w)
    \odot
    \gset(z)
    \quad \text{and} \quad
    \ddot{A}(z,w)
    =
    A(z,w)
    \odot
    \idset(z).
\]
Reversely, the Exponential GFs of graphs and digraphs,
given in \cref{th:classic:gfs},
are used to convert Graphic and Implication GFs
back to Exponential GFs as follows
\[
    A(z,w)
    =
    \widehat{A}(z,w)
    \odot
    G(z)
    \quad \text{and} \quad
    A(z,w)
    =
    \ddot{A}(z,w)
    \odot
    D(2z).
\]
\end{proposition}

\begin{proof}
We present the proof of the first equality,
the other three having similar proofs.
By definition, we have
\[
    \widehat A(z,w) =
    \sum_{n \geq 0} \dfrac{a_n(w)}{(1+w)^{\binom{n}{2}}} \dfrac{z^n}{n!}
    \qquad \text{and} \qquad
    A(z,w) = \sum_{n \geq 0} a_n(w) \frac{z^n}{n!}.
\]
Using the exponential Hadamard product
(with respect to $z$, as always in this paper),
the first expression is decomposed as
\[
    \widehat A(z,w) =
    \bigg(
    \sum_{n \geq 0} a_n(w)
    \dfrac{z^n}{n!}
    \bigg)
    \odot
    \bigg(
    \sum_{n \geq 0}
    \dfrac{1}{(1+w)^{\binom{n}{2}}}
    \dfrac{z^n}{n!}
    \bigg),
\]
where we recognize
\[
    \widehat A(z,w) =
    A(z,w) \odot \gset(z).
\]
\end{proof}

\begin{remark}
In Robinson's seminal paper~\cite{Robinson71},
a conversion operator
\[
    \Delta \left(
        \sum_{n \geq 0} a_n(w) \dfrac{z^n}{n!}
    \right)
    :=
    \sum_{n \geq 0} \dfrac{a_n(w)}{(1+w)^{\binom{n}{2}}} \dfrac{z^n}{n!}
\]
is used instead of the Hadamard product.
In~\cite{DigraphWindow} and~\cite{flajolet2004airy}
it is shown how to represent this operation
using a version of Fourier integral,
which later turns to be helpful
in the asymptotic analysis of these expressions.
However, the inverse operation \( \Delta^{-1} \)
may potentially lead to everywhere divergent series,
which still constitutes a challenge for the asymptotic analysis
of the expressions involving \( \Delta^{-1} \).
\end{remark}

    \subsection{Combinatorial operations}
    \label{sec:combinatorial:operations}

The central idea behind the use of GFs
is that combinatorial operations on the families
translate into analytic operations on their GFs.
To illustrate this concept,
let us look at the disjoint union.

Recall that the GF associated to a family $\mA$
is a formal series of the form
\[
    A(z) = \sum_{a \in \mA} \kappa_{|a|}(z),
\]
where $|a|$ is a non-negative integer number computed from $a$
and called the \emph{size} of \( a \),
and $(\kappa_n(z))_{n \geq 0}$ is a sequence of functions.
Consider two combinatorial families $\mA$ and $\mB$,
assumed disjoint, and their union $\mC = \mA \uplus \mB$.
Then, by definition,
\[
    C(z) =
    \sum_{c \in \mA \uplus \mB} \kappa_{|c|}(z)
    =
    \sum_{a \in \mA} \kappa_{|a|}(z)
    +
    \sum_{b \in \mB} \kappa_{|b|}(z)
    = A(z) + B(z).
\]
Thus, the disjoint union is translated into a sum of GFs.
This holds for all types of GFs.

The next most natural analytic operation on GFs is the product.
The type of GFs used depends on the operation on combinatorial families
that the product of GFs will translate.
In the following paragraphs,
we present this operation
for the various types of GFs used in the paper.

    \subsubsection{Exponential GFs}
    \label{sec:exponential:gf}

Excellent introductions to this topic are provided
in~\cite{flajolet2009analytic,bergeron1998combinatorial}.
We reproduce here the minimal definitions needed for this article.

\paragraph{Labels.}
There are two different frameworks for enumerating graphs:
the \emph{labeled} and \emph{unlabeled} variants
(cf.~\cite{HararyPalmer}).
In the second paradigm, the graphs
are enumerated up to automorphisms.
The purpose behind vertex labeling
is to consider the enumeration problem in its purest form,
without having to take automorphisms into account.
Consequently, the unlabeled versions of the enumerating problems
have been naturally considered as a further step,
and they require the introduction
of the cycle index series~\cite{bergeron1998combinatorial}
and more general enumeration recurrences.
For example, in their papers on digraph enumeration~\cite{liskovets17contribution,robinson1977counting},
Liskovets and Robinson have extended their recurrences
to the unlabeled case, 
which leads to more tedious computations.

All graphs and digraphs considered in this article are \emph{labeled},
meaning that if a graph $G$ contains $n$ vertices,
which we denote by $|G| = n$,
then each vertex from this graph carries a distinct integer
in the set of labels $\{1, \ldots, n\}$.
In \cref{fig:labeled:graph}, we depict a graph \( G \)
with a sequence of labels \( (1, 2, 3, 4) \).
If we replaced those labels
with \( (3, 2, 1, 4) \) or \( (1, 4, 3, 2) \),
the graph would be the same,
while replacing them with \( (2, 1, 3, 4) \)
would produce a different graph.

\begin{figure}[hbt!]
\begin{minipage}[t]{.48\textwidth}
    \begin{center}
        \begin{tikzpicture}[>=stealth',thick, node distance=1.0cm]
\draw
node[arnBleuGrande](a) at ( 0, 0)  { 1 }
node[arnBleuGrande](s) at ( 1,-1)  { 2 }
node[arnBleuGrande](d) at ( 2, 0)  { 3 }
node[arnBleuGrande](f) at ( 1, 1)  { 4 }
;
\sedge{a}{d}[][4];
\sedge{a}{f}[][4];
\sedge{s}{a}[][4];
\sedge{d}{s}[][4];
\sedge{f}{d}[][4];
\end{tikzpicture}
    \end{center}
    \caption{\label{fig:labeled:graph}Example of a labeled graph.}
\end{minipage}\hfill%
\begin{minipage}[t]{.48\textwidth}
    \begin{center}
        \begin{tikzpicture}[>=stealth',thick, node distance=1.0cm]
\draw
node[arnBleuGrande](a) at ( 0, 0)  { 2 }
node[arnBleuGrande](s) at ( 1,-1)  { 5 }
node[arnBleuGrande](d) at ( 2, 0)  { 6 }
node[arnBleuGrande](f) at ( 1, 1)  { 10 }
;
\sedge{a}{d}[][4];
\sedge{a}{f}[][4];
\sedge{s}{a}[][4];
\sedge{d}{s}[][4];
\sedge{f}{d}[][4];
\draw
node[arnOrangeGrande](z)  at (0.3 + 3,   1)  { 1 }
node[arnOrangeGrande](x)  at (0.3 + 3,   0)  { 3 }
node[arnOrangeGrande](q)  at (0.3 + 3.8, 0)  { 4 }
node[arnOrangeGrande](w)  at (0.3 + 3,  -1)  { 7 }
node[arnOrangeGrande](e)  at (0.3 + 4.8,-1)  { 8 }
node[arnOrangeGrande](r)  at (0.3 + 4.8, 0)  { 9 }
node[arnOrangeGrande](t)  at (0.3 + 3.8, 1)  { 11 }
;
\sedge{q}{w}[][4];
\sedge{q}{r}[][4];
\sedge{e}{r}[][4];
\sedge{r}{t}[][4];
\sedge{t}{q}[][4];
\sedge{z}{x}[][4];
\node[rectangle,dashed,draw,fit=(a)(s)(d)(f), very thick,
      rounded corners=5mm,inner sep= 5pt, bgreen] {};
\node[rectangle,dashed,draw,fit=(q)(w)(e)(r)(t), very thick,
      rounded corners=5mm,inner sep= 5pt, bgreen] {};
\end{tikzpicture}
    \end{center}
    \caption{\label{fig:cartesian:product}An element from the labeled product of two 
    certain graph families.}
\end{minipage}
\end{figure}

\paragraph{Labeled product.}
Although they simplify enumeration,
labels introduce the following difficulty.
A pair of labeled graphs is not a labeled object:
indeed, unless one of the graphs has no vertices,
the pair will contain two vertices with label $1$.
This contradicts the requirement that labels are distinct.
There is a classic solution
to solve this issue~\cite[Chapter~II]{flajolet2009analytic}.
As we are going to extend the scheme later
in~\cref{definition:implication:product}, let us recall it here for
completeness.

When forming a pair of labeled (di)graphs \( A \) and \( B \), we introduce
the following \emph{relabeling} scheme:
\begin{itemize}
    \item Assume that \( A \) has \( k \) vertices, and that the
        disjoint union of \( A \) and \( B \) has \( n \) vertices (\ie
        \( B \) has \( n-k \) vertices);
    \item An arbitrary partition \( \sigma_A \uplus \sigma_B = \{ 1,
        \ldots, n \} \), \( |\sigma_A| = k \), \( |\sigma_B| = n-k \)
        is chosen;
    \item The nodes of the graphs \( A \) and \( B \) receive,
        respectively, the labels from \( \sigma_A \) and \( \sigma_B \),
        preserving the relative ordering of the labels within \( A \)
        and \( B \).
\end{itemize}

For example, the left graph from~\cref{fig:cartesian:product}
is a relabeling of the graph from~\cref{fig:labeled:graph}.
The \emph{labeled product} $\mC$
of two labeled families $\mA$ and $\mB$
then contains, for all $a \in \mA$ and $b \in \mB$,
the pairs of relabeled elements $(a', b')$
such that the labels of the pair are
$\{1, 2, \ldots, |a| + |b|\}$
(so the pair is properly labeled).
In \cref{fig:cartesian:product} we depict an element
of the labeled product of two graphs
which carries \( 11 \) labels on its vertices.
If we imagine that the first and the second graphs
belong to some hypothetical families \( \mA \) and \( \mB \),
then, inside one of the resulting relabeled pairs,
the first graph from this pair receives labels
\( (2, 5, 6, 10) \), and the second one receives the remaining ones.
These labels are then arranged in an increasing order
to replace the original ones.
Strictly speaking, the resulting object is not a graph:
its vertices are partitioned into two sets, which,
informally speaking, correspond to a graph
with a marked subset of vertices.

The Exponential GF has been designed precisely
to capture the labeled product as an algebraic operation.
Specifically, let $\mC$ denote
the labeled product of $\mA$ and $\mB$,
and let $c_n$, $a_n$ and $b_n$ denote
the respective number of objects of size $n$.
Following the construction, we obtain
\[
    c_n =
    \sum_{k=0}^n
    \binom{n}{k} a_k b_{n-k}.
\]
Let $C(z)$, $A(z)$ and $B(z)$ denote
the associated Exponential GFs, then
\begin{align*}
    C(z) &=
    \sum_{n \geq 0} c_n \frac{z^n}{n!}
    =
    \sum_{n \geq 0}
    \sum_{k=0}^n
    \binom{n}{k} a_k b_{n-k} \frac{z^n}{n!}
    \\&=
    \sum_{n \geq 0}
    \sum_{k=0}^n
    a_k \frac{z^k}{k!}
    b_{n-k} \frac{z^{n-k}}{(n-k)!}
    =
    \bigg( \sum_{k \geq 0} a_k \frac{z^k}{k!} \bigg)
    \bigg( \sum_{k \geq 0} b_k \frac{z^k}{k!} \bigg)
    =
    A(z) B(z).
\end{align*}
Thus, the Exponential GF of the labeled product
of two families is equal to the product of their Exponential GFs.

\paragraph{Other operations.}
The definition of relabeling extends naturally
to more than two objects.
Consider a labeled family $\mA$
and the family $\mB$ obtained by taking the labeled product
of $\mA$ with itself $k$ times.
Thus, $\mB$ contains sequences of $k$ relabeled objects from $\mA$.
Then $B(z) = A(z)^k$.
Now let us identify two such sequences
if one can be obtained from the other
by changing the order of its elements.
This corresponds to considering
sets of $k$ elements from $\mA$ instead of sequences.
For each sequence, there are $k!$ corresponding sets.
Let $\mC$ denote the family containing
the sets of $k$ (relabeled) elements from $\mA$,
this implies that its Exponential GF $C(z)$
satisfies $A(z)^k = k! C(z)$, so
\[
    C(z) = \frac{A(z)^k}{k!}.
\]
Let $\mD$ denote the family containing all sets
of relabeled elements from $\mA$.
This is the disjoint union
of sets of $k$ elements, for $k \geq 0$.
Since the disjoint union translates into a sum,
we deduce that the Exponential GF of $\mD$ is equal to
\[
    D(z) =
    \sum_{k \geq 0}
    \frac{A(z)^k}{k!} = e^{A(z)}.
\]
Thus, the combinatorial operation \emph{set} is translated,
in the generating functions, by the exponential.

The following classical result illustrates
the power of those simple constructions.

\begin{proposition}
Let $G(z) = \sum_{n \geq 0} (1+w)^{\binom{n}{2}} \frac{z^n}{n!}$
(variable $w$ omitted)
denote the Exponential GF of all graphs
(see \cref{th:classic:gfs}),
then the Exponential GF of \emph{connected graphs} is
\[
    C(z) = \log(G(z)).
\]
\end{proposition}

\begin{proof}
Since a graph is a set of connected components,
their Exponential GFs are linked by the relation
\[
    G(z) = e^{C(z)}.
\]
Inverting this relation gives the announced result.
\end{proof}

    \subsubsection{Graphic GFs} \label{sec:graphic:gf}

The \emph{arrow product} (see \cite{DigraphWindow,de2019symbolic})
$\mC$ of two digraph families \( \mA \) and \( \mB \)
consists of all ordered relabeled digraph pairs
\( (A, B) \) from the labeled product of \( \mA \) and \( \mB \),
equipped with any additional subset of arcs from \( A \) to \( B \).
Let $a_n(w)$, $b_n(w)$, $c_n(w)$
denote the Ordinary GFs associated to the families $\mA$, $\mB$, $\mC$
as in \cref{def:counting:digraphs},
then this construction implies
\[
    c_n(w) =
    \sum_{k=0}^n
    \binom{n}{k}
    (1+w)^{k (n-k)}
    a_k(w) b_{n-k}(w).
\]
The factor $\binom{n}{k}$ comes from the possible relabelings
and the factor $(1+w)^{k (n-k)}$ accounts
for the possible arcs added from
a digraph $A$ (with $k$ vertices) to
a digraph $B$ (with $n-k$ vertices).

The Graphic GFs (introduced by Robinson \cite{Robinson71}
and further refined by Gessel \cite{Gessel95}) have been designed
to capture this convolution rule.
Indeed, denoting by $\widehat A(z,w)$, $\widehat B(z,w)$, $\widehat C(z,w)$
the Graphic GFs corresponding to the families $\mA$, $\mB$, $\mC$,
we have
\begin{align*}
    \widehat C(z,w) &=
    \sum_{n \geq 0}
    \frac{c_n(w)}{(1+w)^{\binom{n}{2}}}
    \frac{z^n}{n!}
    =
    \sum_{n \geq 0}
    \sum_{k=0}^n
    \binom{n}{k}
    (1+w)^{k (n-k) - \binom{n}{2}}
    a_k(w) b_{n-k}(w)
    \frac{z^n}{n!}
    \\&=
    \sum_{n \geq 0}
    \sum_{k=0}^n
    \binom{n}{k}
    \frac{a_k(w)}{(1+w)^{\binom{k}{2}}}
    \frac{z^k}{k!}
    \frac{b_{n-k}(w)}{(1+w)^{\binom{n-k}{2}}}
    \frac{z^{n-k}}{(n-k)!}
    =
    \widehat A(z,w) \widehat B(z,w).
\end{align*}
Thus, the Graphic GF of the arrow product of two digraph families
is the product of their Graphic GFs.

To illustrate the power of this construction,
the next proposition gives the exact enumeration
of strongly connected digraphs (SCCs).
Ideas from this proof and the result itself
will be used in the proofs of \cref{sec:implication:gf}.

\begin{proposition}[{See~\cite{Robinson71} or \cite{de2019symbolic}}]
\label{th:SCC}
The Exponential GF of strongly connected digraphs (components) $\SCC(z,w)$
is equal to
\begin{equation} \label{eq:scc}
    \SCC(z, w) = - \log \left( G(z, w) \odot \frac{1}{G(z, w)} \right),
\end{equation}
where $G(z,w)$ denotes the Exponential GF of all graphs,
from \cref{th:classic:gfs},
and $\odot$ is the exponential Hadamard product,
from \cref{sec:hadamard}.
\end{proposition}

\begin{proof}
The various SCCs of a digraph are disjoint
and each vertex belongs to an exactly one SCC,
so the SCCs form a partition of the vertices.
We say that an SCC $A$ is \emph{source-like}
if there is no arc starting in another SCC
and ending in a vertex of $A$.
Let $\widehat D(z,w,u)$ denote the Graphic GF of all digraphs,
where an additional variable $u$
marks the source-like SCCs.
Then $\widehat D(z,w,v+1)$ is the Graphic GF of all digraphs, 
where an arbitrary subset of source-like SCCs
are marked by the variable $v$.
This family has a unique decomposition
as the arrow product of a set of SCCs
(the source-like SCCs marked by $v$)
with an arbitrary digraph.
The Exponential GF of a set of SCCs marked by $v$ is
\[
    e^{v\, \SCC(z,w)}.
\]
Using \cref{proposition:conversion} to translate,
the Graphic GF of this family is
\[
    \gset(z,w) \odot e^{v\, \SCC(z,w)}.
\]
According to \cref{th:classic:gfs},
the Graphic GF of all digraphs is equal to $G(z,w)$.
Since the arrow product translates into a product
of Graphic GFs, we deduce
\[
    \widehat D(z,w,v+1) =
    \left(
        \gset(z,w) \odot e^{v\, \SCC(z,w)}
    \right)
    G(z,w).
\]
At $v=-1$, the left hand-side, $\widehat D(z,w,0)$,
is the Graphic GF of digraphs that contain no source-like SCC.
The only such digraph is the empty digraph
(containing no vertex), which Graphic GF is $1$, so
\[
    1 =
    \left(
        \gset(z,w) \odot e^{- \SCC(z,w)}
    \right)
    G(z,w).
\]
Solving this equation and using \cref{proposition:conversion}
to translate the Graphic GF into Exponential GF,
we deduce
\[
    \SCC(z,w) =
    - \log \left( G(z,w) \odot \frac{1}{G(z,w)} \right).
\]
\end{proof}

    \subsubsection{Implication GFs}
    \label{sec:implication:gf}

In this section, we present the \emph{Implication product},
a combinatorial operation combining
a \emph{digraph} family with an \emph{implication digraph} family.
We also show that the Implication GF of the Implication product
of two families is the product of their respective
Graphic and Implication GFs.
We have designed Implication GFs
 precisely to ensure this correspondence.
Enumerative results on various $2$-SAT families
will be derived in the next section.

Recall that according to our convention,
the nodes of an implication digraph form a set
\[
    \{  1, 2, \ldots, n, \n 1, \ldots, \n n \},
\]
where \( \n 1, \n 2, \ldots \) denote the negated literals.

\begin{definition}
    \label{definition:implication:product}
Let \( \mD \) be a digraph family and let \( \mF \) be an implication digraph family.
The \emph{implication product} of \( \mD \) and \( \mF \) is formed in the following way.
Let \( D \in \mD \) and \( F \in \mF \) be arbitrary members of these families,
and let \( D \) contain \( k \) vertices and
\( F \) contain \( 2(n-k) \) vertices,
so that the total number of vertices in the union of \( F \), \( D \) and \( \n D \)
is \( 2n \).
\begin{enumerate}
    \item An arbitrary partition of labels \( \sigma_F \uplus \sigma_D
    = \{ 1, \ldots, n \} \), \( |\sigma_F| = k \), \( |\sigma_D| = n-k \)
    is chosen.
\item The nodes of \( D \) and \( F \) respectively receive labels
    from \( \sigma_D \) and \( \sigma_F \) (in the case of \( F \) the
        labels extend to negated literals).
        An arbitrary subset of nodes in $D$ are then labeled as negated
        literals. The resulting (``left'') digraph is called \( L(D) \).
        The negated (``right'') digraph \( \n{L(D)} \)
    is then called \( R(D) \).
    \item A new implication digraph is formed by taking an ordered union of the digraphs
    \( L(D) \), \( F \) and \( R(D) \) with new node labels according to the partition
    \( (\sigma_F, \sigma_D) \).
    \item An arbitrary subset of arcs from \( L(D) \) to \( F \) is added. For each
    \( x \in L(D) \) and \( v \in F \), if an arc \( x \to v \) was added, then
    an arc \( \n v \to \n x \) is also added from \( F \) to \( R(D) \).
    \item An arbitrary subset of arcs from \( L(D) \) to \( R(D) \) is added, ensuring
    that no arc of type \( x \to \n x \) is picked, as this would lead, by symmetry,
    to multiple arcs. If an arc \( x \to y \) is added, then a symmetrical arc
    \( \n y \to \n x \), also from \( L(D) \) to \( R(D) \), is added.
\end{enumerate}
By taking the union over all pairs \( (D, F) \),
all possible label partitions \( (\sigma_D, \sigma_F) \),
vertex negations, and all arc subsets from \( L(D) \) to \( F \) and
from \( L(D) \) to \( R(D) \),
we obtain the implication product of \( \mD \) and \( \mF \).
\end{definition}

\begin{example}
This construction is illustrated in \cref{fig:arrow:product}.
In our example, we let \( k = 4 \) and \( n = 7 \). The digraph \( D \)
receives labels \( \sigma_D = \{ 1, 4, 6, 7 \} \).
Then, according to the arbitrary choices, a vertex with the label \( 6 \)
is labeled as negated, which yields a digraph \( L(D) \).
The negation of \( L(D) \) now has labels \( \{ \n 1, \n 4, 6, \n 7 \} \),
and its arc directions are reversed.
The implication digraph \( F \) receives the remaining labels \( \{2, 3, 5\} \).
Then, an arbitrary subset of arcs is added from \( L(D) \) to \( F \),
which is, in our example, a set
\( \{ 7 \to 2, \n 6 \to 2 \} \), and, by symmetry, there is a subset of arcs
\( \{ \n 2 \to \n 7, \n 2 \to 6\} \) from \( F \) to \( R(D) \).
Finally, the arcs \( \{ 1 \to \n 4, 4 \to \n 1 \} \) are added from \( L(D) \) to
\( R(D) \) in a way that avoids adding arcs \( x \to \n x \), and preserves the
symmetry property of implication digraphs.
\end{example}

\begin{figure}[hbt!]
    \begin{center}
        \begin{tikzpicture}[>=stealth',thick, scale = 0.8]
\draw
node[arn_b](x) at ( -1 + 0, -.2)  {\( 1 \)}
node[arn_b](y) at ( -1 + 1.2,-1)  {\( 4 \)}
node[arn_r](z) at ( -1 + 2, 0.2)  {\( \fn 6 \)}
node[arn_b](t) at ( -1 + 0.8, 1)  {\( 7 \)}
;
\draw
node[arn_b](a)  at (3, 1)  {\( 2 \)}
node[arn_b](b)  at (3.3, -.2)  {\( 3 \)}
node[arn_b](c)  at (3,-1)  {\( 5 \)}
node[arn_r](a') at (5, 1)  {\( \fn 2 \)}
node[arn_r](b') at (5,-1)  {\( \fn 3 \)}
node[arn_r](c') at (4.5, -.2)  {\( \fn 5 \)}
;
\draw
node[arn_r](x') at (7 + 0.4, -1)  {\( \fn 1 \)}
node[arn_r](y') at (7 + 2,-.6)  {\( \fn 4 \)}
node[arn_b](z') at (7 + 1.6, 1)  {\( 6 \)}
node[arn_r](t') at (7 + 0, 0.6)  {\( \fn 7 \)}
;
\fromto{x}{y};
\fromto{y}{z};
\fromto{z}{t};
\fromto{t}{x};
%
\fromto{y'}{x'};
\fromto{z'}{y'};
\fromto{t'}{z'};
\fromto{x'}{t'};
\fromto{b}{c'};
\fromto{c}{b'};
\fromto{a}{b};
\fromto{b'}{a'};
\node[rectangle,dashed,draw,fit=(x)(y)(z)(t), very thick,
      rounded corners=5mm,inner sep= 5pt, bgreen] {};
\node[rectangle,dashed,draw,fit=(x')(y')(z')(t'), very thick,
      rounded corners=5mm,inner sep= 5pt, bgreen] {};
\node[rectangle,dashed,draw,fit=(a)(a')(c)(c'), very thick,
      rounded corners=5mm,inner sep= 5pt, bviolet] {};
\path (y) edge [blackred,thick, bend right=35,
decoration={markings,mark=at position 0.99 with
{\arrow[ultra thick,blackred, rotate=0]{>}}}, postaction={decorate}
] node {} (x');
\path (x) edge [blackred,thick, bend right=55,
decoration={markings,mark=at position 0.99 with
{\arrow[ultra thick,blackred, rotate=0]{>}}}, postaction={decorate}
] node {} (y');
\path (z) edge [blackred,thick, bend right=10,
decoration={markings,mark=at position 0.99 with
{\arrow[ultra thick,blackred, rotate=0]{>}}}, postaction={decorate}
] node {} (a);
\path (a') edge [blackred,thick, bend left=55,
decoration={markings,mark=at position 0.99 with
{\arrow[ultra thick,blackred, rotate=0]{>}}}, postaction={decorate}
] node {} (z');
\aprod{t}{a};
\aprod{a'}{t'};
\node at (-1, -2) {$L(D)$, $D \in \mathcal D$};
\node at (4, -2) {$F \in \mathcal F$};
\node at (9.2, -2) {$R(D)$, $D \in \mathcal D$};
\end{tikzpicture}
    \end{center}
    \caption{An element from the implication product of a certain digraph family
    \( \mD \) and a certain implication digraph family \( \mF \).}
    \label{fig:arrow:product}
\end{figure}
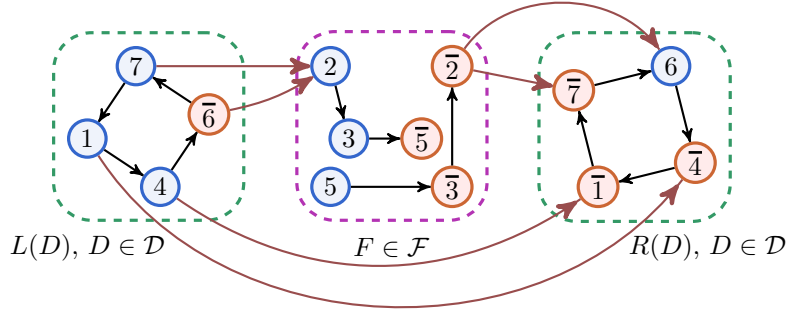

Let us emphasize that, by definition, this product operation is not commutative,
because it involves two families of different kinds,
namely the digraphs and the implication digraphs.
This explains why we need two separate kinds of generating functions,
one for the digraph family,
the other one for the implication digraph family.
Surprisingly, in the context of this combinatorial operation,
there is no need to introduce a new type of GF for digraphs as we
can still use the Graphic GF.
%

\begin{proposition} \label{th:implication:product:from:gf}
Let \( \mA \) be a family of digraphs and
\( \mB \) be a family of implication digraphs,
and let \( \widehat A(z,w) \) and \( \ddot{B}(z,w) \)
denote their respective Graphic and Implication GFs.
Let \( \mC \) denote their implication product,
and \( \ddot{C}(z, w) \) its Implication GF.
Then,
\[
    \ddot{C}(z, w) = \widehat A(z, w) \cdot \ddot{B}(z, w).
\]
\end{proposition}

\begin{proof}
    Let us construct the convolution rule corresponding to the implication product.
    Let
    \[
        a_n(w) = \sum_{m \geq 0} a_{n,m} w^m
        \quad \text{and} \quad
        b_n(w) = \sum_{m \geq 0} b_{n,m} w^m,
    \]
    where \( a_{n,m} \) denotes the number of digraphs from \( \mA \) with $n$ nodes and
    $m$ arcs, and \( b_{n,m} \) denotes the number of implication formulae from
    \( \mB \) with $2n$ vertices $2m$ implication arcs (\ie
    corresponding to 2-CNFs with $n$ variables and $m$ clauses).

    Let us compute the generating function \( c_n(w) \)
    of the number of ways to form an implication digraph
    with $n$ vertices in total. Suppose that a digraph \( D \in \mA \)
    has $k$ vertices and a corresponding implication digraph \( F \in \mB \)
    has \( 2(n-k) \) vertices.
    We need to take a sum over all possible values of $k$.
    The number of ways to choose the labels belonging to either side of the product is
    \( \binom{n}{k} \). Then, there are \( 2^k \) ways to choose a negated subset of
    vertices in the digraph \( L(D) \).
    Next, the generating function of the number of ways
    to choose a subset of arcs from
    \( L(D) \) to \( F \) is \( (1+w)^{k \cdot 2(n-k)} \).
    Finally, drawing the edges from \( L(D) \) to \( R(D) \) yields a choice
    from \( k(k-1)/2 \) possible combinations (each edge has also a
    complementary negated edge, which provides the factor $1/2$).
    This gives in total
    \begin{align*}
        c_n(w) &= \sum_{k = 0}^n \binom{n}{k} 2^k (1+w)^{2k(n-k) + k(k-1)/2} a_k(w) b_{n-k}(w)
        \\&=
        \sum_{k = 0}^n
        \binom{n}{k}
        2^n
        (1+w)^{n (n - 1)}
        \frac{a_k(w)}{(1+w)^{\binom{k}{2}}}
        \frac{b_{n-k}(w)}{(1+w)^{(n-k)(n-k-1)} 2^{n-k}}.
    \end{align*}  
    On the other hand, by expanding the brackets in
    \( \widehat A(z, w) \cdot \ddot{B} (z, w) \), we also obtain, by grouping the summands,
    \begin{align*}
        \widehat A(z, w) \cdot \ddot{B} (z, w)
        &=
        \sum_{n \geq 0}
        \sum_{k=0}^n
        \frac{a_k(w)}{(1+w)^{\binom{k}{2}}}
        \frac{b_{n-k}(w)}{(1+w)^{(n-k)(n-k-1)} 2^{n-k}}
        \frac{z^n}{n!}
        \\&=
        \sum_{n \geq 0} \dfrac{c_n(w)}{(1+w)^{n(n-1)}} \dfrac{z^n}{2^n n!},
    \end{align*}
    which completes the proof.  
\end{proof}

\begin{remark}
A similar technique can be used if edges of the form \( x \to \n x \) are allowed,
or when loops and multiple edges are allowed. In such models it is useful to recall
that the composition operation requires dealing with
\emph{compensation factors}~\cite{janson1993birth}, which was handled in
\cite{panafieu2016analytic} using a very natural construction of GF
which is doubly-exponential in variables marking both vertices and edges.
By further exploring this idea with 2-SAT, it is possible to
arrive at a similar definition of the compensation factor for such formulae in the
presence of multiple arcs and loops, similar
to~\cite{DigraphWindow,dovgal2019birth}.
\end{remark}

\section{Counting 2-SAT families}
\label{sec:counting:sat:families}

In this Section, we use the implication product
to obtain the GFs of satisfiable 2-CNFs
and contradictory strongly connected components,
as well as 2-CNFs whose implication digraphs
have prescribed ordinary and contradictory SCCs.

\subsection{The main decomposition scheme}

The first proposition we introduce
exposes a link between the generating function of all 2-CNFs
and all digraphs.
It will be used to simplify the expressions where they appear.
Recall that the Exponential GF of the implication digraphs given
in~\cref{def:counting:digraphs} is not constrained to even powers:
the counting sequence is indexed by \( n \) and \( m \), where
\( n \) denotes \emph{half} the number of vertices, and \( m \)
denotes \emph{half} the number of arcs. This convention seems natural
if one considers CNF as sets of clauses with Boolean variables, but it can
be less intuitive when manipulating directed combinatorial structures
such as implication digraphs.

\begin{proposition} \label{th:CNF:D}
Let \( \CNF(z) = \CNF(z,w) \)
denote the Exponential GF of all implication digraphs,
then its corresponding Implication GF is
\[
    \ddot{\CNF}(z) = D(z/2).
\]
\end{proposition}

\begin{proof}
A 2-CNF on $n$ variables is characterized by its set of clauses.
The set of all possible clauses has cardinality $4 \binom{n}{2} = 2n(n-1)$,
so the Exponential GF of 2-CNFs is
\[
    \CNF(z) =
    \sum_{n \geq 0}
    (1+w)^{2n(n-1)} \frac{z^n}{n!}.
\]
By application of \cref{proposition:conversion},
the corresponding Implication GF is then
\[
    \ddot{\CNF}(z) =
    \sum_{n \geq 0}
    \frac{(1+w)^{2n(n-1)}}{(1+w)^{n(n-1)}}
    \frac{z^n}{2^n n!}
    =
    \sum_{n \geq 0}
    (1+w)^{n(n-1)}
    \frac{(z/2)^n}{n!},
\]
which is equal to $D(z/2)$.
\end{proof}

Now, let us explore the structural properties of a formula coming from an
implication digraph.
Recall that according to \cref{th:SCC}, $\SCC(z) = -\log(G(z) \odot G(z)^{-1})$
denotes the Exponential GF of strongly connected digraphs.

\begin{lemma} \label{th:cnf:u:v}
Let \( \CNF(z, u, v) \) (variable $w$ omitted)
denote the Exponential GF of all implication digraphs,
where \( u \) marks the number of \emph{source-like
non-isolated ordinary} strongly connected components,
and \( v \) marks the number of unordered tuples \( \{C, \n C\} \)
containing an \emph{isolated ordinary} component \( C \) and its
negation \( \n C \)
    (in the sense of \cref{th:ordinary:SCC:symmetric}, in other words,
    \( v \) marks twice the number of isolated ordinary components).
Let $\ddot{\CNF}(z)$ denote the Implication GF of implication digraphs
(from \cref{th:CNF:D}), then
\begin{equation} \label{eq:cnf:u:v}
    \CNF(z, s+1, 2s+2t+1)
    =
    \left(\left[
        (e^{s\, \SCC(z)} \odot \gset(z))
        \ddot{\CNF}(z)
    \right] \odot D(2z) \right)
    e^{t\, \SCC(2z)}.
\end{equation}
\end{lemma}

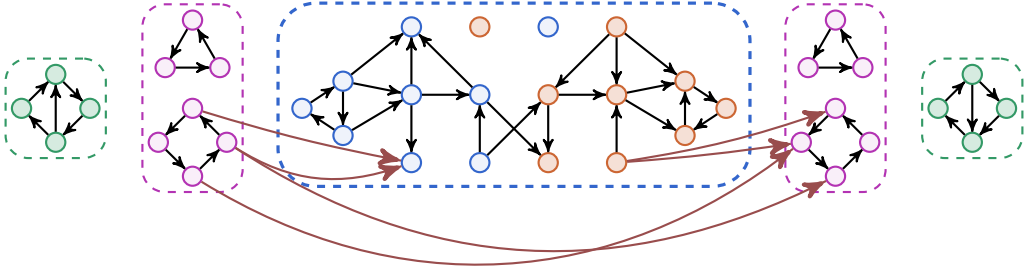
\begin{figure}[hbt!]
    \begin{center}
        \begin{tikzpicture}[>=stealth',thick, scale = 0.9]
\draw
%
node[arnVertPetit,scale=0.9](iA) at ( -0.2, -.7)   {}
node[arnVertPetit,scale=0.9](iS) at (  0.3, -.2)   {}
node[arnVertPetit,scale=0.9](iD) at ( -0.2,  .3)   {}
node[arnVertPetit,scale=0.9](iF) at ( -0.7, -.2)   {}
node[arnVertPetit,scale=0.9](iA') at ( 13 + 0.2, -.7)   {}
node[arnVertPetit,scale=0.9](iS') at ( 13 - 0.3, -.2)   {}
node[arnVertPetit,scale=0.9](iD') at ( 13 + 0.2,  .3)   {}
node[arnVertPetit,scale=0.9](iF') at ( 13 + 0.7, -.2)   {}
node[arnVioletPetit,scale=0.9] (aA) at (1.2+.5 + 0.1, 1.1)  {}
node[arnVioletPetit,scale=0.9] (aS) at ( .8+.5 + 0.1,  .4)  {}
node[arnVioletPetit,scale=0.9] (aD) at (1.6+.5 + 0.1,  .4)  {}
node[arnVioletPetit,scale=0.9](dA) at ( 1.3+.5, -1.2)   {}
node[arnVioletPetit,scale=0.9](dS) at ( 1.8+.5,  -.7)   {}
node[arnVioletPetit,scale=0.9](dD) at ( 1.3+.5,  -.2)   {}
node[arnVioletPetit,scale=0.9](dF) at (  .8+.5,  -.7)   {}
%
node[arnBleuPetit,scale=0.9] (sA)  at (4.0, .2) {}
node[arnBleuPetit,scale=0.9] (sS)  at (4.0,-.6) {}
node[arnBleuPetit,scale=0.9] (sD)  at (3.4,-.2) {}
node[arnBleuPetit,scale=0.9](q)  at ( 5, 0)  { }
node[arnBleuPetit,scale=0.9](w)  at ( 5,-1)  { }
node[arnBleuPetit,scale=0.9](e)  at ( 6,-1)  { }
node[arnBleuPetit,scale=0.9](r)  at ( 6, 0)  { }
node[arnBleuPetit,scale=0.9](t)  at ( 5, 1)  { }
node[arnRougePetit,scale=0.9](y)  at ( 6, 1)  { }
node[arnRougePetit,scale=0.9] (sA')  at (13 - 4.0, .2) {}
node[arnRougePetit,scale=0.9] (sS')  at (13 - 4.0,-.6) {}
node[arnRougePetit,scale=0.9] (sD')  at (13 - 3.4,-.2) {}
node[arnRougePetit,scale=0.9](q')  at ( 13 - 5, 0)  { }
node[arnRougePetit,scale=0.9](w')  at ( 13 - 5,-1)  { }
node[arnRougePetit,scale=0.9](e')  at ( 13 - 6,-1)  { }
node[arnRougePetit,scale=0.9](r')  at ( 13 - 6, 0)  { }
node[arnRougePetit,scale=0.9](t')  at ( 13 - 5, 1)  { }
node[arnBleuPetit,scale=0.9](y')  at ( 13 - 6, 1)  { }
node[arnVioletPetit,scale=0.9] (aA') at ( 13 - 1.7 - 0.1, 1.1) {}
node[arnVioletPetit,scale=0.9] (aS') at ( 13 - 1.3 - 0.1, .4)  {}
node[arnVioletPetit,scale=0.9] (aD') at ( 13 - 2.1 - 0.1, .4)  {}
node[arnVioletPetit,scale=0.9](dA') at ( 13 - 1.8, -1.2)   {}
node[arnVioletPetit,scale=0.9](dS') at ( 13 - 2.3,  -.7)   {}
node[arnVioletPetit,scale=0.9](dD') at ( 13 - 1.8,  -.2)   {}
node[arnVioletPetit,scale=0.9](dF') at ( 13 - 1.3,  -.7)   {}
;

\node[rectangle,dashed,draw,
fit=(q)(w)(e)(r)(t)(sA)(sS)(sD)(q')(w')(e')(r')(t')(sA')(sS')(sD'),
      rounded corners = 5mm,inner sep = 5pt, bblue, very thick] {};
\node[rectangle,dashed,draw,fit=(dA)(dS)(dD)(dF)(aA)(aS)(aD),
      rounded corners = 3mm,inner sep = 2pt, bviolet, thick] {};
\node[rectangle,dashed,draw,fit=(dA')(dS')(dD')(dF')(aA')(aS')(aD'),
      rounded corners = 3mm,inner sep = 2pt, bviolet, thick] {};
\node[rectangle,dashed,draw,fit=(iA)(iS)(iD)(iF),
      rounded corners = 3mm,inner sep = 2pt, bgreen, thick] {};
\node[rectangle,dashed,draw,fit=(iA')(iS')(iD')(iF'),
      rounded corners = 3mm,inner sep = 2pt, bgreen, thick] {};

\fromto{q}{w};
\fromto{q}{r};
\fromto{e}{r};
\fromto{r}{t};
\fromto{q}{t};
\fromto{aA}{aS};
\fromto{aS}{aD};
\fromto{aD}{aA};
\fromto{dA}{dS};
\fromto{dS}{dD};
\fromto{dD}{dF};
\fromto{dF}{dA};
\fromto{aS'}{aA'};
\fromto{aD'}{aS'};
\fromto{aA'}{aD'};
\fromto{dS'}{dA'};
\fromto{dD'}{dS'};
\fromto{dF'}{dD'};
\fromto{dA'}{dF'};
\fromto{sA}{sS};
\fromto{sS}{sD};
\fromto{sD}{sA};
\fromto{sA}{t};
\fromto{sA}{q};
\fromto{sS}{q};
\fromto{w'}{q'};
\fromto{r'}{q'};
\fromto{r'}{e'};
\fromto{t'}{r'};
\fromto{t'}{q'};
\fromto{sS'}{sA'};
\fromto{sD'}{sS'};
\fromto{sA'}{sD'};
\fromto{t'} {sA'};
\fromto{q'} {sA'};
\fromto{q'} {sS'};
\fromto{e}{r'};
\fromto{r}{e'};
\fromto{iS}{iA};
\fromto{iA}{iD};
\fromto{iD}{iS};
\fromto{iF}{iD};
\fromto{iA}{iF};
\fromto{iA'}{iS'};
\fromto{iD'}{iA'};
\fromto{iS'}{iD'};
\fromto{iD'}{iF'};
\fromto{iF'}{iA'};

\path (dD) edge [blackred,thick, bend right=3,
    decoration={markings,mark=at position 0.99 with
    {\arrow[ultra thick,blackred, rotate=0]{>}}}, postaction={decorate}
    ] node {} (w);
\path (dS) edge [blackred,thick, bend right=26,
    decoration={markings,mark=at position 0.99 with
    {\arrow[ultra thick,blackred, rotate=0]{>}}}, postaction={decorate}
    ] node {} (w);
\path (w') edge [blackred,thick, bend right=5,
    decoration={markings,mark=at position 0.99 with
    {\arrow[ultra thick,blackred, rotate=0]{>}}}, postaction={decorate}
    ] node {} (dD');
\path (w') edge [blackred,thick, bend right=2,
    decoration={markings,mark=at position 0.99 with
    {\arrow[ultra thick,blackred, rotate=0]{>}}}, postaction={decorate}
    ] node {} (dS');
\path (dA) edge [blackred,thick, bend right=35,
    decoration={markings,mark=at position 0.995 with
    {\arrow[ultra thick,blackred, rotate=0]{>}}}, postaction={decorate}
    ] node {} (dS');
\path (dS) edge [blackred,thick, bend right=30,
    decoration={markings,mark=at position 0.995 with
    {\arrow[ultra thick,blackred, rotate=0]{>}}}, postaction={decorate}
    ] node {} (dA');

\end{tikzpicture}
    \end{center}
    \caption{A schematic representation of the main decomposition scheme
    for $\CNF(z, s+1, 2s + 2t + 1)$ (labels omitted for convenience).
    Mirror symmetry reflects a relation between pairs of negated literals.
    A 2-CNF implication digraph without any marked components is located
    in the center (in blue and red).
    On the sides are the marked source-like (and sink-like) distinguished
    components (second leftmost and second rightmost, in purple).
    Some of these components may happen to be isolated.
    Then come marked isolated components, also in pairs (leftmost
    and rightmost, in green).}
    \label{fig:main:decomposition}
\end{figure}

\begin{proof}
Recall that according to \cref{th:ordinary:SCC:symmetric},
for any isolated ordinary component $C$ of an implication digraph,
the digraph also contains an isolated ordinary component $\n C$
obtained from $C$ by negating the literals
and reversing the arcs.
Let $\mF$ denote the family of implication digraphs where
\begin{itemize}
\item
a subset of non-isolated source-like ordinary components
are marked by the variable $s$,
\item
a subset of all unordered tuples $\{C, \n C\}$
of isolated ordinary components is distinguished.
In each of those tuples, one component is chosen
and marked either by the variable $s$,
or by the variable $t$.
\end{itemize}
The Exponential GF of $\mF$ is then $\CNF(z, s+1, 2s + 2t + 1)$,
which is the left hand-side of~\eqref{eq:cnf:u:v}.
Any implication digraph in $\mF$ has a unique decomposition as
\begin{itemize}
\item[(i)]
a set of \emph{ordered} tuples $(C, \n C)$ of isolated ordinary components
(the components $C$ are the one marked by $t$),
\item[(ii)]
and the implication product of
a \emph{set} of ordinary components
(they correspond to the components marked by $s$)
with an arbitrary implication digraph
(the unmarked part of the implication digraph).
\end{itemize}
This decomposition is depicted in \cref{fig:main:decomposition}.
We now show that it translates into the generating function given
in the right hand-side of~\eqref{eq:cnf:u:v}.

\noindent \textbf{Item (i).}
\cref{th:ordinary:SCC:symmetric} gives a recipe
to build an ordered tuple of isolated ordinary components.
\begin{enumerate}
\item
Start with a strongly connected digraph (component) $B$;
\item
For each vertex $x$, choose to keep it as a literal $x$,
or replace it with its negation $\n x$.
We denote the resulting component by $C$;
\item
Add a negated component $\n C$
(obtained by negating each literal and reversing the arcs).
\end{enumerate}
The ordered tuple is then $(C, \n C)$.
This construction implies that
the Exponential GF of ordered tuples of isolated ordinary components
is $\SCC(2 z)$.
Thus, the Exponential GF of sets
of \emph{ordered} tuples of isolated ordinary components,
marked by the variable $t$, is $e^{t\, \SCC(2 z)}$.

\noindent \textbf{Item (ii).}
The Exponential GF of a set of ordinary components marked by $s$
is $e^{s\, \SCC(z)}$.
Applying \cref{proposition:conversion},
its Graphic GF is $e^{s\, \SCC(z)} \odot \gset(z)$.
By \cref{th:implication:product:from:gf},
the Implication GF of the implication product
from Item (ii) is
\[
    \left( e^{s\, \SCC(z)} \odot \gset(z) \right)
    \ddot \CNF(z)
\]
so, according to \cref{proposition:conversion},
its Exponential GF is
\[
    \left[
        \left( e^{s\, \SCC(z)} \odot \gset(z) \right)
        \ddot \CNF(z)
    \right]
    \odot D(2z)
\]
Combining (i) and (ii) in a labeled product
(see \cref{sec:exponential:gf}),
we deduce that the Exponential GF of $\mF$ is
\[
    \left(
        \left[
            \left( e^{s\, \SCC(z)} \odot \gset(z) \right)
            \ddot \CNF(z)
        \right]
        \odot D(2z)
    \right)
    e^{t\, \SCC(2 z)}.
\]
\end{proof}

Furthermore, we can consider the case when the contradictory and
ordinary strongly connected components of the implication digraph only
belong to the two given families.
This allows us to obtain the Exponential GF of these implication
digraphs.

\begin{lemma} \label{th:cnf:scc:cscc:u:v}
Let \( \CNF_{\scc,\cscc}(z, u, v) \) (variable $w$ omitted)
denote the Exponential GF of implication digraphs
whose ordinary SCCs belong to the family $\scc$
and whose contradictory SCCs belong to the family $\cscc$,
where \( u \) marks the number of source-like
non-isolated ordinary strongly connected components, and \( v \) marks the
    number of unordered tuples \( \{C, \n C\} \)
    containing an isolated ordinary component \( C \) and its negation
    \( \n C \)
(in the sense of \cref{th:cnf:u:v,th:ordinary:SCC:symmetric}).
Then, the following decomposition is valid:
\begin{equation}
\label{main:eq}
    \CNF_{\scc,\cscc}(z, s+1, 2s+2t+1)
    =
    \left(\left[
        (e^{s\, \scc(z)} \odot \gset(z))
        \ddot{\CNF}_{\scc,\cscc}(z,1,1)
    \right] \odot D(2z) \right)
    e^{t\, \scc(2z)},
\end{equation}
where \( \scc(z) \) is the Exponential GF of the family \( \scc \).
\end{lemma}

The proof of the lemma is identical to the proof of the previous one.
Now, we want to obtain the generating function \( \cscc(z) \),
which is the Exponential GF of the family \( \cscc \), based on the
previous result.
Note that there is no generating function of the family \( \cscc \) of any type
entering the previous expression.
In order to solve the previous equation and identify \( \CNF_{\scc,\cscc}(z, u, v) \),
we need to use another combinatorial property of the implication digraphs
which results in an additional initial condition when \( u = 0 \).
The following result is thus independent
from \cref{th:cnf:u:v} and \cref{th:cnf:scc:cscc:u:v}.

\begin{lemma}
With the previous notation, we have
    \begin{equation}
    \label{eq:no:source:like}
        \CNF_{\scc,\cscc}(z, 0, v)
        =
        e^{\cscc(z) + v\, \scc(2z)/2},
    \end{equation}
where $\cscc(z)$ is the Exponential GF of the family \( \cscc \).
\end{lemma}
\begin{proof}
The left expression is the Exponential GF of the implication digraphs
where all the source-like ordinary strongly connected components are
isolated, and where unordered tuples of isolated ordinary components are
    marked by \( v \).
Let us temporarily remove all the isolated ordinary components from the
implication digraph.
Clearly, if a digraph is not empty, there should be at least one
    source-like SCC.
Since all the source-like ordinary SCCs are now removed,
it should be a contradictory one.
But according to \cref{th:structure:CSCC}, all the source-like
    contradictory SCCs should be isolated.
Therefore, after returning back the removed isolated ordinary
components, an implication digraph corresponding to such 2-CNF
is decomposed into a set of disjoint contradictory SCCs
    and unordered tuples of isolated ordinary SCCs (each marked by $v$).
Now, the \( \mathsf{SET} \) operation on implication digraphs is again expressed
using the exponential function, and the Exponential GF of one pair of isolated
components from \( \scc \) is \( \scc(2z)/2 \) because the pair is non-ordered.
This yields the expression for the generating function.
\end{proof}

Finally, by combining the previous two lemmas,
we arrive at the enumeration formula for all implication digraphs
whose ordinary and contradictory SCCs belong to given families.
Furthermore, fixing the allowed families allows even more flexible analysis
by weighting the elements of these families and by using those weights
as additional marking parameters in order to count the number of
specific types of components in a formula, which we shall see later.

\begin{proposition} \label{th:cnf:scc:cscc}
Let
\( \ddot{\CNF}_{\scc,\cscc}(z) \) be the Implication GF
of the implication digraphs whose Exponential GFs of allowed
ordinary and contradictory SCCs are, respectively, $\scc(z)$ and
    $\cscc(z)$, then
\begin{equation} \label{eq:chosen:cscc:scc}
    \ddot{\CNF}_{\scc,\cscc}(z)
    =
    \dfrac{
        e^{\cscc(z) - \scc(2z)/2}
        \odot \idset(z)
    }{
        e^{- \scc(z)} \odot \gset(z)
    }.
\end{equation}
\end{proposition}

\begin{proof}
By plugging \( s = -1 \) into~\eqref{main:eq}, we obtain
\begin{equation}
    \CNF_{\scc,\cscc}(z, 0, 2t-1)
    =
    \left(
    \left[
        (e^{- \scc(z)} \odot \gset(z))
        \ddot{\CNF}_{\scc,\cscc}(z)
    \right]
    \odot D(2z)
    \right)
    e^{t\, \scc(2z)}
\end{equation}
and by combining with~\eqref{eq:no:source:like}, we obtain
\begin{equation}
    e^{\cscc(z) + (2t-1) \scc(2z)/2}
    =
    \left(
    \left[
        (e^{- \scc(z)} \odot \gset(z))
        \ddot{\CNF}_{\scc,\cscc}(z)
    \right]
    \odot D(2z)
    \right)
    e^{t\, \scc(2z)},
\end{equation}
which, at $t=0$ (or, in fact, with any $t$), yields
\begin{equation}
    e^{\cscc(z) - \scc(2z)/2}
    \odot \idset(z)
    =
    (e^{- \scc(z)} \odot \gset(z))
    \ddot{\CNF}_{\scc,\cscc}(z).
\end{equation}
This completes the proof.
\end{proof}

\subsection{Counting satisfiable 2-CNFs and contradictory SCCs}

A first application of~\cref{th:cnf:scc:cscc}
is the enumeration of satisfiable 2-CNFs.

\begin{theorem} \label{th:satisfiable}
Let $\ddot{\SAT}(z)$ denote the Implication GF of satisfiable 2-CNFs.
Then,
\[
    \ddot{\SAT}(z) =
    G(z) \left(
    \sqrt{G(z) \odot \frac{1}{G(z)}}
    \odot \idset(2z) \right).
\]
\end{theorem}

\begin{proof}
A formula is satisfiable if and only if its set of contradictory SCCs is empty.
Injecting $\cscc(z) = 0$ and $\scc(z) = \SCC(z)$ into~\eqref{eq:chosen:cscc:scc}
gives
\[
    \ddot{\SAT}(z) =
    \frac{e^{-\SCC(2z)/2} \odot \idset(z)}
        {e^{-\SCC(z)} \odot \gset(z)}.
\]
The Hadamard relation
\[
    f(a z) \odot g(z) = f(z) \odot g(a z)
\]
is applied in the numerator
\begin{equation} \label{eq:ddotsat}
    \ddot{\SAT}(z) =
    \frac{e^{-\SCC(z)/2} \odot \idset(2 z)}
        {e^{-\SCC(z)} \odot \gset(z)}.
\end{equation}
Replacing the generating function $\SCC(z)$
with its expression from~\eqref{eq:scc},
we have
\[
    e^{-\SCC(z)} \odot \gset(z) =
    G(z) \odot \frac{1}{G(z)} \odot \gset(z).
\]
Given the expressions of $G(z)$ and $\gset(z)$
provided in \cref{proposition:conversion},
the exponential Hadamard product
$G(z) \odot \gset(z)$ is equal to $1$, so
the denominator of~\eqref{eq:ddotsat} is
\[
    e^{-\SCC(z)} \odot \gset(z) =
    \frac{1}{G(z)}.
\]
The numerator is expressed as the Hadamard product with the square root.
\end{proof}

Note that exponential Hadamard product can become an obstacle in a potential
future asymptotic analysis of satisfiable 2-CNF due to undefined behavior
of divergent series.
To assist in this journey, we propose another formulation of the last result
containing fewer Hadamard products, but including a sort of large power coefficient extraction \cite[Theorem VIII.8]{flajolet2009analytic}.

\begin{theorem} \label{th:variant:satisfiable}
The number of satisfiable 2-CNFs with $n$ variables and $m$ clauses is
\[
    2^n n! [z^n w^m]
    \gset((1 + w)^{2(n-1)} z, w)
    \sqrt{G(z,w) \odot \frac{1}{G(z,w)}}.
\]
\end{theorem}

\begin{proof}
Let $A(z)$ denote the function
\[
    A(z) = \sqrt{G(z) \odot \frac{1}{G(z)}}.
\]
From~\cref{th:satisfiable},
the number $\SAT_{n,m}$ of satisfiable 2-CNFs
with $n$ variables and $m$ edges is
\[
    \SAT_{n,m} =
    n! [z^n w^m]
    D(2z) \odot \left(
    G(z) \left(
    A(z)
    \odot \idset(2z) \right) \right).
\]
Replacing $D(2z)$ and $\idset(2z)$ by their expressions
from~\cref{proposition:conversion}
and extracting the $[z^n]$ coefficient, we obtain
\begin{align*}
    \SAT_{n,m} &=
    2^n
    [w^m]
    (1+w)^{n(n-1)}
    \sum_{k=0}^n
    \binom{n}{k}
    (1+w)^{\binom{n-k}{2}}
    \frac{k! [z^k] A(z)}{(1+w)^{k(k-1)}}
    \\&=
    2^n
    [w^m]
    \sum_{k=0}^n
    \binom{n}{k}
    (1+w)^{\binom{n-k}{2} + n(n-1) - k(k-1)}
    k! [z^k] A(z) .
\end{align*}
Rewriting the power of $(1+w)$ as $4 \binom{n}{2} - \binom{n-k}{2} - 2(n-1) k$,
we obtain
\begin{align*}
    \SAT_{n,m} &=
    2^n
    [w^m]
    (1+w)^{4 \binom{n}{2}}
    \sum_{k=0}^n
    \binom{n}{k}
    (1+w)^{- \binom{n-k}{2}}
    k! [z^k] A((1+w)^{-2(n-1)} z, w)
    \\&=
    2^n
    n! [z^n w^m]
    (1+w)^{4 \binom{n}{2}}
    \gset(z) A((1+w)^{-2(n-1)} z, w)
    \\&=
    2^n
    n! [z^n w^m]
    \gset((1+w)^{2(n-1)} z) A(z).
\end{align*}
\end{proof}

The second implication of~\cref{th:cnf:scc:cscc} is
the enumeration of contradictory SCCs.

\begin{theorem} \label{th:CSCC}
The Exponential GF of contradictory strongly connected
implication digraphs (components) is given by
\[
    \CSCC(z)
    =
    \dfrac{1}{2} \SCC(2z)
    +
    \log \left(
    D(z) \odot 
        \dfrac{D(z)}{G(2z)}
    \right).
\]
\end{theorem}

\begin{proof}
    By applying~\cref{th:cnf:scc:cscc}, we obtain
    \[
        \CSCC(z)
        =
        \dfrac{1}{2} \SCC(2z)
        +
        \log \left(
        D(2z) \odot \left[
            \left(
                e^{-\SCC(z)} \odot \gset(z)
            \right)
            \ddot{\CNF}(z)
        \right]
        \right)
    \]
    where $\ddot{\CNF}(z)$ is equal to $D(z/2)$
    according to~\cref{th:CNF:D}.
    Applying the property \( e^{-\SCC(z)} = G(z) \odot G(z)^{-1}\)
    and the Hadamard property
    $A(2z) \odot B(z) = A(z) \odot B(2z)$ finishes the proof.
\end{proof}

Finally, the most detailed description of implication digraphs with
marked parameters including source-like components, isolated components and
marked contradictory SCCs summarizes several of the previous results.

\begin{theorem} \label{th:detailed:CNF}
Let $\CNF(z, u, v)$ denote the Exponential GF
of the implication digraphs whose Exponential GFs of allowed
ordinary and contradictory SCCs are, respectively, $\scc(z)$ and $\cscc(z)$,
and the variables \( z \), \( u \) and \( v \) mark, respectively,
the vertices, non-isolated source-like ordinary SCCs
    and unordered tuples of isolated ordinary components (in the sense
    of~\cref{th:cnf:u:v}),
then
\begin{align*}
    &\CNF_{\scc, \cscc}(z, u, v)
    \\&=
    \left(
    \left[
        \dfrac{
            \left( e^{(u-1) \scc(z)} \odot \gset(z) \right)
            \left( e^{\cscc(z)-\scc(2z)/2} \odot \idset
            (z)
            \right)
        }{
            e^{-\scc(z)} \odot \gset(z)
        }
    \right]
    \odot D(2z)
    \right)
    e^{(v+1-2u)\scc(2z)/2}.
\end{align*}
\end{theorem}

\begin{proof}
We start with~\eqref{main:eq} from \cref{th:cnf:scc:cscc:u:v}.
The variable change
\[
    (s, t) = \left(u-1, \frac{v+1-2u}{2} \right)
\]
is applied
\begin{equation} \label{eq:cnf:scc:cscc:u:v:partial}
    \CNF_{\scc,\cscc}(z, u, v)
    =
    \left(\left[
        \left( e^{(u-1) \scc(z)} \odot \gset(z) \right)
        \ddot{\CNF}_{\scc,\cscc}(z,1,1)
    \right] \odot D(2z) \right)
    e^{(v+1-2u) \scc(2z)/2}.
\end{equation}
We fix $u=0$ and solve with respect to $\ddot{\CNF}_{\scc,\cscc}(z,1,1)$,
using \cref{proposition:conversion} to reverse the exponential Hadamard
    products:
\[
    \ddot{\CNF}_{\scc,\cscc}(z,1,1)
    =
    \frac{
        \left(
            \CNF_{\scc,\cscc}(z, 0, v)
            e^{- (v+1) \scc(2z)/2}
        \right)
        \odot
        \idset(z)
    }{
        e^{- \scc(z)} \odot \gset(z)
    }
    \, .
\]
The expression of $\CNF_{\scc,\cscc}(z, 0, v)$
from~\eqref{eq:no:source:like} is injected:
\[
    \ddot{\CNF}_{\scc,\cscc}(z,1,1)
    =
    \frac{
        e^{\cscc(z) - \scc(2z)/2}
        \odot
        \idset(z)
    }{
        e^{- \scc(z)} \odot \gset(z)
    }
    \, .
\]
The result of the theorem is obtained
by injecting this last equation into~\eqref{eq:cnf:scc:cscc:u:v:partial}.
\end{proof}

\begin{remark}
In \cref{th:detailed:CNF}, we could introduce an additional variable $q$
marking the contradictory SCCs as well.
To do so, simply replace the generating function $\cscc(z)$
with $q \, \cscc(z)$:
\begin{align*}
    &\CNF_{\scc, \cscc}(z, u, v, q)
    \\&=
    \left(
    \left[
        \dfrac{
            \left( e^{(u-1) \scc(z)} \odot \gset(z) \right)
            \left( e^{q\, \cscc(z)-\scc(2z)/2} \odot \idset
            (z)
            \right)
        }{
            e^{-\scc(z)} \odot \gset(z)
        }
    \right]
    \odot D(2z)
    \right)
    e^{(v+1-2u)\scc(2z)/2}.
\end{align*}
The same applies to ordinary SCC.
\end{remark}

\section{Numerical results}
\label{sec:numerical}

The first few coefficients of the sequences enumerating satisfiable 2-CNF
and contradictory SCC are given in~\cref{table:sat:1,table:sat:2,table:cscc}.
Using exhaustive generation techniques, we have verified that our
enumeration scheme gives correct answers for all 2-CNF families with at
most \( 4 \) Boolean variables and up to \( 7 \) clauses.
It is still worth treating some specific examples by hand.

\begin{table}[hbt!]
\centering
\caption{\label{table:sat:1}Enumerating satisfiable 2-SAT formulae. Here, $a_n$ is the counting sequence
of satisfiable formulae, where $n$ denotes the number of Boolean variables. We provide also
 $a_{n,m}$ where $m$ is the number of clauses.}
\scalebox{0.85}{ 
\begin{tabular}[t]{ c r r r r r r r }
\toprule
 $n$ & $a_n$        & $m=0$ & $m=1$ & $m=2$ & $m=3$ & $m=4$ & $m=5$ \\
\midrule
 1 & 1 &                      1 & 0  & 0    & 0     & 0       & 0        \\
 2 & 15 &                     1 & 4  & 6    & 4     & 0       & 0        \\
 3 & 2397 &                   1 & 12 & 66   & 220   & 486     & 684      \\
 4 & 3049713 &                1 & 24 & 276  & 2024  & 10596   & 41616    \\
 5 & 28694311447 &            1 & 40 & 780  & 9880  & 91320   & 654408   \\
 6 & 2034602766692687 &       1 & 60 & 1770 & 34220 & 487500  & 5451072  \\
 7 & 1115068294703296663717 & 1 & 84 & 3486 & 95284 & 1929270 & 30847236 \\
\bottomrule
\end{tabular}}
\end{table}

As a first check, we observe that the coefficients $a_{n,\,m}$ are those of $[w^m](1+w)^{2n(n-1)}$ for all
$m \leq 3$. Indeed in expectation three quarters of the clauses are SAT during a
random assignment by a greedy algorithm of the variables. The
probabilistic method~\cite{PROBABILISTIC} tells us then that for all $m \leq 3$,
there is only one integer value that exceeds the expected number of
satisfied clauses by the greedy algorithm, and this
 unique integer is the value $m$.

As a second check, the first discrepancy is given by $a_{3,\,4} = 486$ where as there are in total
$[w^4](1+w)^{12} = 495$ $2$-CNF formulae built with $3$ variables and $4$ clauses. The $9$ UNSAT formulae
built with the variables $x, \, y, \, z$ and $4$ clauses can be deduced
by considering all the permutations of literals
\( \{x, \n x, y, \n y, z, \n z\} \)
from the constructions given
in~\eqref{eq:2NDCHECK}, where
$6$ formulae come from the first construction and $3$ come from the second:
\\
\begin{equation}\label{eq:2NDCHECK}
\begin{tabular}{c c c}
  $  \begin{cases}
    x \vee y\\
    x \vee \n y\\
    \n x \vee z \\
    \n x \vee \n z
  \end{cases}
  $
  & $\qquad \qquad$ & 
  $   \begin{cases}
    x \vee y\\
    x \vee \n y \\
    \n x \vee y \\
    \n x \vee \n y
  \end{cases}
  $
\end{tabular}
\end{equation}

\begin{table}[hbt!]
\centering
\caption{\label{table:sat:2}Enumerating satisfiable 2-SAT formulae, continuation.}
\scalebox{0.85}{
\begin{tabular}[t]{ c r r r r r r r }
\toprule
 $n$ & $m=6$    & $m=7$ & $m=8$ & $m=9$ & $m=10$ & $m=11$ & $m=12$ \\
\midrule
 3 & 572       & 276        & 72          & 8            & 0   & 0     & 0      \\
 4 & 123528    & 275568     & 463680      & 596232       & 593928  & 462408   & 281896    \\
 5 & 3752600   & 17428040   & 65774970    & 202646120    & 514203264  & 1087043720   & 1937000920   \\
 6 & 49675760  & 377136960  & 2411974740  & 13063104000  & 60169952412 & 237115483560  & 805717285720  \\
 7 & 405181084 & 4485339276 & 42527890314 & 348648091120 & 2484665216376 & 15453747532944 & 84253905879486 \\
\bottomrule
\end{tabular}}
\end{table}

\begin{table}[hbt!]
\centering
\caption{\label{table:cscc}Enumerating contradictory strongly connected 2-CNF.
Here, $a_n$ is their counting sequence, where $n$ denotes the number of Boolean variables.
 $k$ denotes the excess of a component and is equal
to  its number of clauses minus its number of variables.}
\scalebox{0.8}{
\begin{tabular}[t]{ c r r r r r r r }
\toprule
 $n$ & $a_n$        & $k=1$ & $k=2$ & $k=3$ & $k=4$ & $k=5$ & $k=6$ \\
\midrule
 2 & 1 &                          0 & 1  & 0    & 0     & 0       & 0        \\
 3 & 1606 &                       6 & 84 & 316   & 492   & 417     & 212      \\
 4 & 12864042 &                   144 & 4104 & 38880  & 186864  & 559496   & 1175064    \\
 5 & 1035697286504 &              2880 & 152160 & 2779350  & 26769440  & 165382784   & 733763440   \\
 6 & 1137724245192445576 &        57600 & 5097600 & 157060200 & 2572386420 & 27182781120  & 207149446560  \\
 7 & 19275699325699284398997808 & 1209600 & 166199040 & 7932622320 & 201117551040 & 3285880363290 & 38654632189488 \\
\bottomrule
\end{tabular}
}
\end{table}

Below, we provide empirical time measurements
(performed on a 2014 MacBook Air, Intel Core
i5 with 1,4 GHz) for computing the
total number of satisfiable 2-SAT formulae with different numbers of
variables as an indication (see~\cref{table:time:sat,table:time:sat:m}),
by using the generating functions that we provide in the paper.
In the case with two parameters the exact calculation take
much longer due to necessity of considering bivariate series.

\begin{table}[hbt!]
\centering
\caption{\label{table:time:sat}Empirical time measurements for computing
    the total number of satisfiable 2-SAT formulae.}
\scalebox{0.8}{
\begin{tabular}[t]{ l r r r r r r r r r r r r}
\toprule
    $n$ & $10$ & $20$ & $40$ & $100$ & $200$ & $250$ & $350$ & $500$
    & $600$ & $700$ & $800$ & $900$ \\
\midrule
    Time & 0.3ms & 0.5ms & 16ms & 200ms & 2s & 3.76s & 14s & 45s
    & 1m34s & 2m30s & 3m55s & 6m52s \\
\bottomrule
\end{tabular}
}
\end{table}

\begin{table}[hbt!]
\centering
\caption{\label{table:time:sat:m}Empirical time measurements for computing
    the number of satisfiable 2-SAT formulae with given number of
    clauses. Decimal precision for interval arithmetic picked up
    heuristically.}
\scalebox{0.8}{
\begin{tabular}[t]{ l r r r r r r r r r }
\toprule
    $n$ & $10$ & $20$ & $25$ & $30$ & $35$ & $40$ & $45$ & $50$
    & $60$ \\
\midrule
    $m=n$ & 27ms  & 340ms & 870ms & 1.7s & 4s & 12.8s & 22.7s & 41.6s & 2min\\
    $m=2n$ & 50ms & 1.16s & 2.8s & 7.17s & 22.3s & 48.4s & 1m39s & 2m20s
    & 4m7s\\
\bottomrule
\end{tabular}
}
\end{table}

Recall that $\proba_{\SAT}(n,p)$ denotes the probability
for a random $(n,p)$ 2-CNF to be satisfiable,
and its limit probability in the critical window
\( p=(1 + \mu n^{-1/3}) / (2n) \) is denoted by
\[
    \proba_{\SAT, \infty}(\mu) =
    \lim_{n \to +\infty}
    \proba_{\SAT} \left(
        n,
        \frac{1}{2 n} \left(1 + \mu n^{-1/3} \right)
    \right).
\]
Using the data from generating functions with moderate values of $n$ we can provide
very precise although non-rigorous estimates for $\proba_{\SAT, \infty}(\mu)$. In~\cite{Deroulers}, Deroulers and Monasson used Monte-Carlo
simulation approach to empirically estimate the limit probability $\proba_{\SAT, \infty}(0)$ that a random 2-SAT
with clause probability \( p = \frac{1}{2n} \) is satisfiable,
which corresponds to the center of the critical window
of the phase transition.
Using the assumption that the limiting probability behaves as
\( \proba_{\SAT} \left(n, \frac{1}{2 n} \right) \sim \proba_{\SAT, \infty}(0) - c n^{-1/3} \),
they empirically estimated the
coefficient $c$ using linear regression which lead them to an estimate
\[
    \proba_{\SAT, \infty}(0) = 0.907 \pm 10^{-3},
\]
by using the data obtained from various values of $n$ up to $n = 5 \times 10^6$.
Although this assumption is very plausible by taking into account the analogy
with digraphs and random graphs~\cite{luczak2009critical,FPSAC20,janson1993birth},
it is still an open question,
as far as we know. However, using the same assumption 
 along with the machinery of generating functions, we can provide much
 more accurate predictions by simply taking more terms of the asymptotic
 expansion.

With our method, we do not have to use Monte-Carlo simulation to obtain the
finite-size probabilities with high accuracy: using the interval variant of
long arithmetic and recurrences from generating functions, we can obtain
these probabilities with arbitrarily high precision.
Relying on the asymptotic equivalence of the models \( (n,m) \) and \( (n,p) \)
as \( n \to \infty \),
we can argue that the limiting probabilities do not depend on which model is chosen.
The probabilities in the \( (n,p) \) model can be expressed via generating functions
using~\cref{proposition:proba:np}.
Inside the critical window of the phase transition
\( p = \frac{1}{2n}(1 + \mu n^{-1/3}) \) we use an assumption
that the limiting probability that a random 2-SAT formula with $n$ Boolean variables and
clause probability $p$ is satisfiable, asymptotically behaves as
\begin{equation}
\label{eq:full:asympt}
    \proba_{\SAT} \left(n, \frac{1}{2n}(1 + \mu n^{-1/3})\right)
    \sim
    \proba_{\SAT, \infty}(\mu) + c_1(\mu) n^{-1/3} + c_2(\mu) n^{-2/3} + \ldots
\end{equation}
and we can therefore use multidimensional linear regression to estimate the
coefficients \( (c_k(\mu))_{k=1}^\infty \) and $\proba_{\SAT, \infty}(\mu)$ in the setting
where almost no noise is present.
Our estimates for coefficients $c_1(0), \ldots, c_6(0)$
in the center of the critical window are given in~\cref{table:full:asympt}.

By computing these probabilities for only 100 points $n$ in the range from 100 to 5000,
which can be computed in only a few minutes (!),
and by ensuring that no numerical instability is present in our estimates
(also known as ``overfitting''), we predict, by applying linear
regression with dimension $7$,
\[
    \proba_{\SAT, \infty}(0) = 0.90622396067 \pm 10^{-11}.
\]
The estimated error $10^{-11}$ has been obtained by comparing
the regression models with different dimensions and with a different level of ``noise''
truncation, if we consider the measurements with smaller values of $n$ to be more ``noisy''.
These prediction can be improved by taking more points and a higher upper bound.

\begin{table}[hbt!]
\centering
\caption{\label{table:full:asympt}Numerical estimates of the coefficients of the
full asymptotic expansion~\eqref{eq:full:asympt}.}
\begin{tabular}[t]{ c l}
\toprule
$\proba_{\SAT, \infty}(0)$ & \texttt{ 0.90622396067 +/- 1e-11} \\
\midrule
$c_1$ & \texttt{ 0.212314432 +/- 2e-9}  \\
$c_2$ & \texttt{-0.17396477 +/- 2e-8} \\
$c_3$ & \texttt{ 0.066792 +/- 2e-6} \\
$c_4$ & \texttt{ 0.0155 +/- 2e-4} \\
$c_5$ & \texttt{-0.041 +/- 2e-3} \\
$c_6$ & \texttt{ 0.021 +/- 2e-3} \\
\bottomrule
\end{tabular}
\end{table}

By using this method for values of $\mu$ other than zero,
we obtain the predicted plot of
the limiting function \( \proba_{\SAT, \infty}(\mu) \) in the range \( \mu \in [-4, 4] \),
which is shown in~\cref{fig:limiting:plot}.

\section{Conclusion}
Random 2-CNFs are fundamental objects in combinatorics
and analysis of algorithms.
Having exact expressions for 2-SAT formulae potentially opens many new possibilities
to describe the properties of a typical 2-SAT formula.
However, the analytic tools to extract the asymptotic of the coefficients
of such generating functions
are not yet developed. More specifically, the GF
\begin{equation}
\label{eq:scc:conclusion}
    e^{-\SCC(z, w)} = G(z, w) \odot \dfrac{1}{G(z, w)}
\end{equation}
appearing in~\cref{th:satisfiable,th:variant:satisfiable}
is a GF whose coefficients are growing faster than exponentially for any fixed
positive value of \( w \). One of the few tools for dealing with such divergent
series is the Large Powers Theorem \cite[Theorem
VIII.8]{flajolet2009analytic} and its variations, which requires a specific
representation with additional variables.
We expect that representation from~\cref{th:variant:satisfiable}
will be helpful in this direction.
Two other possibilities worth exploring are the formal integral
representation of~\eqref{eq:scc:conclusion}
(see \cite{flajolet2004airy} to grasp the difficulties involved)
or an application of Bender's theorem \cite{Bender75}
since the sequence is growing sufficiently quickly.

Phase transitions are intriguing phenomena,
linking combinatorics, algorithmics and statistical physics.
For example, the random instances of NP-complete problems
that are difficult to solve for heuristics
tend to appear inside the phase transition window
\cite{achlioptas2008algorithmic}.
Although 2-SAT is not a computationally challenging algorithmic problem,
its phase transition has for a long time eluded
the application of existing combinatorial tools.
It has embodied the \emph{simplest unsolved problem}
for various techniques at different times.
Even though the phase transition window
and its width have already been obtained
\cite{goerdt92,chvatalreed92,de2001random,bollobas2001scaling},
to the present day, a combinatorial description
inside the critical window is still missing.
The present paper constitutes a step in that direction.

\paragraph*{Acknowledgements.}
The authors are grateful to Danièle Gardy for her support and encouragement.
Sergey Dovgal was supported by the HÉRA project, funded by The French National
Research Agency, grant no.: ANR-18-CE25-0002.
Élie de Panafieu was supported by the Lincs (\url{www.lincs.fr})
and the Rise project RandNET, grant no.: H2020-EU.1.3.3.
Vlady Ravelomanana is partly supported by the CNRS IRN Project ``Aléa Network''.


\bibliographystyle{acm}
\bibliography{2SAT}


\end{document}